\documentclass[12pt]{amsart}
\usepackage{epsfig,amsmath,amsfonts,latexsym}
\usepackage[abs]{overpic}
\usepackage[usenames,dvipsnames]{color}
\usepackage{palatino}
\usepackage[hypcap=false]{caption}
\usepackage{xcolor}

\usepackage{color}

\usepackage[hyperfootnotes=false]{hyperref}
\hypersetup{
  colorlinks,
  citecolor=Purple,
  linkcolor=Black,
  urlcolor=Green}
  
\theoremstyle{plain}
\newtheorem{dummy}{anything}[section]
\newtheorem{theorem}[dummy]{Theorem}

\newtheorem{lemma}[dummy]{Lemma}

\newtheorem{question}[dummy]{Question}

\newtheorem{proposition}[dummy]{Proposition}

\theoremstyle{definition}
\newtheorem{definition}[dummy]{Definition}

\newtheorem{example}[dummy]{Example}
\newtheorem{construction}[dummy]{Construction}

\newtheorem{remark}[dummy]{Remark}

\theoremstyle{remark}

\newcommand{\del}{\partial}

\newcommand{\Z}{\mathbb{Z}}
\newcommand{\R}{\mathbb{R}}
\newcommand{\C}{\mathbb{C}}

\def\dfn#1{{\em #1}}
\def\R{\mathbb{R}}
\def\Z{\mathbb{Z}}
\def\C{\mathbb{C}}

\begin{document}

\title[Iterated planar contact manifolds]{Generalizations of planar contact manifolds\\ to higher dimensions}

\author{Bahar Acu}
\address{Claremont Colleges: Pitzer College \\ Claremont \\ California \\ USA}
\email{baharacu@gmail.com}
\email{baharacu@pitzer.edu}

\author{John B. Etnyre}
\address{School of Mathematics \\ Georgia Institute of Technology \\  Atlanta  \\ Georgia \\ USA}
\email{etnyre@math.gatech.edu}

\author{Burak Ozbagci}
\address{Department of Mathematics \\ Koc University \\ Istanbul \\ Turkey}
\email{bozbagci@ku.edu.tr}



\begin{abstract} Iterated planar contact manifolds are a generalization of three dimensional planar contact manifolds to higher dimensions. We study some basic topological properties of iterated planar contact manifolds and discuss several examples and constructions showing that many contact manifolds are iterated planar. We also observe that for any odd integer $m > 3$, any finitely presented group can be realized as the fundamental group of some iterated planar contact manifold of dimension $m$. Moreover, we introduce another generalization of three dimensional planar contact manifolds that we call projective. Finally, building symplectic cobordisms via open books, we show that some projective contact manifolds admit explicit symplectic caps. 
\end{abstract}

\maketitle

\section{Introduction}

The groundbreaking work of  Giroux \cite{Giroux02}, that characterized contact manifolds in terms of open book decompositions with symplectic pages, has had a remarkable impact in the study of the global topology of contact manifolds. By studying the properties of the pages of an open book, one can learn a lot about the adapted contact structure. For example, given a contact $3$--manifold $(M, \xi)$, one may ask what the minimal genus is for an open book that supports $\xi$. This is the {\em support genus} $sg(\xi)$ of $\xi$,  \cite{EtnyreOzbagci08}. If $sg(\xi)=0$, then the contact structure $\xi$ and  the contact $3$--manifold $(M, \xi)$, as well as the supporting open book, are all called planar. So far there is a great deal known about planar contact manifolds (in dimension 3, by definition), which we itemize below.  

\begin{enumerate}
\item\label{1} All overtwisted contact structures are planar and hence a contact structure with positive support genus is tight, \cite{Etnyre04b}. 
\item\label{2}  Any (weak) symplectic filling $(X,\omega)$ of a planar contact manifold must have connected boundary, and $b_2^+(X)=b_2^0(X)=0$, \cite{Etnyre04b}.
\item\label{3}  In particular, by the proof of item (2) in  \cite{Etnyre04b}, any planar contact manifold has a symplectic cap that contains a symplectic sphere of positive self-intersection and any symplectic filling of a planar contact manifold embeds as a convex domain in a rational surface. 
\item\label{4}  Planar contact manifolds do not embed as weak nonseparating contact-type hypersurfaces in closed symplectic $4$--manifolds (cf. \cite{AlbersBramhamWendl2010, NiederkrugerWendl11}).
\item\label{5}  The Weinstein conjecture is true for planar contact manifolds, \cite{AbbasCieliebakHofer2005}. More generally, it is now known to be true for all contact $3$--manifolds \cite{Taubes07}, but the conjecture was first established for planar contact manifolds, for which the proof is much less involved compared to the general case. 
\item\label{6}  Given a symplectic filling of a contact manifold, a planar open book supporting the contact structure can be extended to a Lefschetz fibration over the symplectic filling,  \cite{NiederkrugerWendl11, Wendl10}. 
\item\label{7}  The previous fact allows us to classify fillings of certain planar contact manifolds \cite{Kaloti??, LiOzbagci18, PlamenevskayaVanHorn-Morris2010}. 
\item\label{8}  Any weak symplectic filling of a planar contact manifold can be deformed into a Stein filling \cite{NiederkrugerWendl11, Wendl10}. 
\item\label{9}  Planarity of a contact manifold can be obstructed by the contact element  in Heegaard Floer homology, \cite{OzsvathStipsiczSzabo05}. 
\end{enumerate}

In this article, we discuss two notions that can be thought of as generalizations of planar open books to higher dimensions:  \begin{itemize}
\item iterated planar open books and 
\item projective open books. 
\end{itemize}

We will mainly focus on basic properties and examples of these open books, and more importantly,  the contact manifolds supported by such open books. A long term goal is to see how many of the items above can be proven for these special contact manifolds in higher dimensions. 

We would like to point out that only very recently Colin, Honda, and Tian \cite{ColinHondaTian20}  set up the general framework of higher-dimensional Heegaard Floer homology, defined the contact class, and used it to give an obstruction to the Liouville fillability of a contact manifold and a sufficient
condition for the Weinstein conjecture to hold. We will say nothing more about item (9), and concentrate on the others. 


\subsection{Iterated planar open books and adapted contact manifolds}\label{sec: itpobs}

Iterated planar contact manifolds (which are of dimension greater than three, by definition) were introduced by the first author in \cite{Acu17preprint}, as a natural generalization of planar contact manifolds. Briefly, an open book is called iterated planar, if its page has an iterated planar Lefschetz fibration, and  the notion of iterated planarity for a Lefschetz fibration is defined inductively by saying that its regular fibers are iterated planar, with the base case being a $2$--dimensional regular fiber that is a planar surface. See Section~\ref{sec:IPOB} for a precise definition. A contact manifold supported by an iterated planar open book is called {\em iterated planar.} 

To date only a few of the above results are known for iterated planar contact manifolds. The first substantial result for iterated planar contact manifolds  was the verification of the Weinstein conjecture by the first author.

\begin{theorem}[Acu 2017, \cite{Acu17preprint}]
If $(M, \xi)$ is an iterated planar contact manifold then any Reeb vector field for $\xi$ has a periodic orbit. 
\end{theorem}
Another proof of this result was given by the first author and Moreno, who also proved the following results. 
\begin{theorem}[Acu-Moreno 2018, \cite{AcuMoreno18preprint}]\label{obstr1}
Any symplectic filling of an iterated planar contact manifold must have connected boundary. 
\end{theorem}
\begin{theorem}[Acu-Moreno 2018, \cite{AcuMoreno18preprint}]
An iterated planar contact manifold cannot be a nonseparating weak contact-type hypersurface in any closed symplectic manifold. 
\end{theorem}

Items~\eqref{4}, \eqref{5}, and the first part of Item~\eqref{2} are  known to be true for iterated planar contact manifolds. 

So far there has been little systematic investigation of how common it is for a contact manifold to be iterated planar. In Section~\ref{IPgeneral}, we show that there are iterated planar contact manifolds  in all odd dimensions (greater than three) and we also show that there is no obstruction to iterated planarity coming from the fundamental group.
\begin{theorem}\label{IPgroups}
Let $G$ be a finitely presented group. Then for each $n\geq 2$, there is an iterated planar contact $(2n+1)$--manifold whose fundamental group is $G$.
\end{theorem}

\begin{remark} 
It is clear that for $n \geq 3$, any finitely presented group $G$ can be realized as the fundamental group of a contact $(2n+1)$--manifold, since there exists a closed orientable $(n+1)$--manifold $X$ with $\pi_1(X)=G$ and the unit cotangent bundle $ST^*X$ equipped with its canonical contact structure is a contact  $(2n+1)$--manifold with $$\pi_1(ST^*X) \cong \pi_1(X) \cong G.$$  This argument obviously breaks down for $n=2$. Nevertheless, A'Campo and Kotschick \cite{ACampoKotschick94}  observed that any such $G$  can be realized as the fundamental group of a contact $5$--manifold, via performing surgery on the contact $5$--manifold $\#_k (S^1 \times S^4, \xi_{std})$, where $k$ is the number of generators in the presentation of $G$ and $\xi_{std}$ is induced by the Stein $6$--manifold $S^1 \times B^5$. Notice that Theorem~\ref{IPgroups} gives an alternate uniform proof of these results as well using a different method. 
\end{remark} 

Iterated planarity is preserved under connected sums.
\begin{theorem} \label{IPconnect}
The connected sum of iterated planar contact manifolds is iterated planar. 
\end{theorem}

We now turn our attention to contact $5$--manifolds. We begin by recalling that in general, a subcritical Stein filling of a contact $(2n-1)$--manifold is a Stein filling that does not have handles of index $n$. 

\begin{theorem}\label{IPsubcrit1}
The following subcritically Stein fillable contact $5$--manifolds are iterated planar.
\begin{enumerate}
\item The contact structure on $S^1\times S^4$ filled by a Stein structure on $S^1\times D^5$ is iterated planar.\footnote{This contact structure is the unique one on $S^1\times S^4$ admitting a symplectically aspherical subcritical Stein filling,  see \cite{BarthGeigesZehmisch19}.}
\item Consider a $5$--manifold $M$ with a finite cyclic fundamental group (including the trivial group). If the order of the group is odd, then any   subcritically Stein fillable contact structure on $M$ is iterated planar. If the order of the group is even,  then at least "half" of the subcritically Stein fillable contact structures on $M$ are iterated planar.
\item If $(M, \xi)$ is a subcritically Stein fillable contact $5$--manifold, then $M$ admits a subcritically Stein fillable iterated planar contact structure $\xi'$ in the same homotopy class of almost contact structures. Moreover,  if $M$ has infinite second cohomology then it has infinitely many such contact structures.  
\end{enumerate}
\end{theorem}

\begin{remark} The proof of (2) above, relies on the Ding-Geiges-Zhang classification \cite{DingGeigesZhang18} of subcritically Stein fillable contact $5$--manifolds with finite cyclic fundamental group of order $m$.  In particular, when $m$ is even, there is a unique element of $2$--torsion in $H^2(M; Z)$;  the (nonsurjective) map from homotopy classes of almost contact structures to
their Chern class in $H^2(M;Z)$ is two-to-one, and hence it makes sense to talk about ``half" of the subcritically Stein fillable contact structures on $M$.

We also note that a $5$--manifold $M$ can have a subcritically Stein  fillable contact structure $\xi$ and another contact structure $\xi'$ that is not subcritically Stein fillable. In general,  we can say nothing about the iterated planarity of $\xi'$, but we expect there to be cases where it is not.  For example, let $X$ be the Stein manifold obtained form $B^4$ by attaching a Stein $2$--handle along  the maximal Thurston-Bennequin right-handed Legendrian torus knot in $(S^3, \xi_{std})$, and let $(M,\xi)$ be the boundary of the subcritical Stein manifold $X\times D^2$.  We observe that the open book with page $X$ and trivial monodromy, that supports  $(M, \xi)$, is not iterated planar because otherwise the planar contact manifold $\partial X$ would have a filling $X$ with $b_2^0 (X)=1$, which contradicts the result discussed  in Item~\eqref{2} above. Although this does not yet prove that the contact $5$--manifold  $(M, \xi)$ is not iterated planar, we present some further supporting evidence.  Notice that any symplectically aspherical filling of $(M,\xi)$ must be diffeomorphic to $X\times D^2$ by \cite[Theorem~1.5]{BarthGeigesZehmisch19},  and it is hard to see how any open book induced from a Lefschetz fibration on $X\times D^2$ could be iterated planar.  Thus we conjecture that  $(M, \xi)$ is not iterated planar. One can clearly create infinitely many other such examples as well by attaching a Stein $2$--handle to $B^4$ along any Legendrian knot in $(S^3, \xi_{std})$ with Thurston-Bennequin invariant larger than zero. This discussion leads to the following question. 
\end{remark}

\begin{question}
Are all subcritically Stein fillable contact manifolds iterated planar? Is there some criteria that determines when they are?
\end{question}

We note that in dimension 3, a subcritically Stein fillable contact structure is the connect sum of some number of copies of $S^1\times S^2$ with its unique tight contact structure. This is well-known to be planar, so the above question asks if this observation generalizes to higher dimensions. 

On the other hand, there are many non-subcritically Stein fillable contact manifolds which are iterated planar. 
\begin{theorem}\label{felxibleIP}
If $(M^5, \xi)$ is an iterated planar contact $5$--manifold with $c_1(\xi)=0$, and has a flexible Stein filling, then $M$ admits infinitely many iterated planar contact structures (in the same homotopy class of almost contact
structures) which are Stein  fillable but not subcritically Stein fillable.
\end{theorem}

A Stein manifold is flexible if it admits a handle decomposition with all the critical handles attached along loose Legendrian spheres \cite{Murphy??}. In particular, a subcritical Stein manifold is flexible, by definition. Thus combining Theorem~\ref{IPsubcrit1} (3)  and Theorem~\ref{felxibleIP},  we conclude that if a $5$--manifold $M$ with infinite second cohomology admits a subcritically Stein fillable contact structure with $c_1=0$, then there are infinitely many iterated planar contact structures on $M$  which are not subcritically Stein fillable. There are many $5$--manifolds satisfying these assumptions. For example, if $X$ is the result of attaching a Stein $2$--handle to $B^4$ along a Legendrian knot with rotation number zero,  then the boundary of the subcritical Stein $6$-manifold $X\times D^2$ is a contact $5$--manifold with infinite second cohomology such that $c_1=0$.

We now observe that Item~\eqref{8} above does not hold for iterated planar contact manifolds.
\begin{proposition}\label{fillablebutnoteStein}
There are  iterated planar contact manifolds that are strongly symplectically fillable but not Stein fillable. 
\end{proposition}

Turning to the opposite end of the spectrum we consider overtwisted contact structures.
\begin{theorem}\label{ot}
If $M$ is a $5$--manifold with finite cyclic fundamental group of odd order (including the trivial group) that is the boundary of a subcritical Stein domain, then every overtwisted contact structure on $M$ is iterated planar.

More generally, if $M$ is any $5$--manifold that admits a subcritically fillable contact structure, then it admits an overtwisted contact structure that is iterated planar. If $H^2(M;\Z)$ is infinite, then there are infinitely many homotopy classes of almost contact structures for which the overtwisted contact structure in any of these classes is iterated planar.
\end{theorem}

For example, it follows from the first part of Theorem~\ref{ot} that the unique overtwisted contact structure on $S^5$, and any overtwisted contact structure on $S^2\times S^3$ and $S^2\widetilde{\times} S^3$ (the nontrivial $S^3$--bundle over $S^2$) are  iterated planar, since each of these simply-connected $5$--manifolds is the boundary of a subcritical Stein domain. Similarly, the second part of Theorem~\ref{ot} implies that the {\em unique} overtwisted contact structure on $S^1 \times S^4$  is iterated planar, since  $S^1 \times S^4$  admits a subcritically fillable contact structure. This gives some evidence that the analog of Item~\eqref{1} may be true, so we ask:
\begin{question}
Is every overtwisted contact manifold iterated planar?
\end{question}

Another natural class of $5$--manifolds to consider is the simply-connected ones. In \cite{Barden65}, Barden gave a classification of closed, oriented, smooth, simply-connected $5$--manifolds, up to diffeomorphism. Note that the vanishing of the third integral Stiefel-Whitney class $W_3(M)$ is a necessary condition for the existence of an almost contact structure on a (not necessarily simply-connected) $5$--manifold $M$, and Massey \cite{Massey61} showed that it is also sufficient. Moreover, Geiges \cite{Geiges91} proved that if $M$ is simply-connected, then it admits a contact structure in every homotopy class of almost contact structures. These results imply that any simply-connected  $5$--manifold which carries a contact structure can be uniquely decomposed into the connected sum of prime manifolds $M_k$ for $1 \leq k \leq \infty$, with possibly one extra summand $X_{\infty}=S^2 \widetilde{\times} S^3$. Here $M_1=S^5$, $M_{\infty}=S^2 \times S^3$, and $H_{2}(M_k; \Z)=\Z_k \oplus \Z_k$ for $1 \leq k < \infty$. 

The results above imply that all overtwisted contact structures on $M_1=S^5$, $M_\infty=S^2\times S^3$, and $X_\infty=S^2 \widetilde{\times} S^3$ are iterated planar and each of these manifolds has infinitely many Stein fillable contact structures that are iterated planar. 

\begin{question}
Are all the contact structures on $S^5$, $S^2\times S^3$, and $S^2 \widetilde{\times} S^3$ iterated planar?
\end{question}
\begin{question} \label{ques: mk} For $2\leq k<\infty$, does the manifold $M_k$  admit any iterated planar contact structures?
\end{question}

We discuss Question~\ref{ques: mk} further in Section~\ref{IP5D} and suggest that $M_k$ admits a contact structure that we conjecture is not iterated planar. 

Many of the other items listed above for planar contact manifolds follow from finding symplectic caps for these manifolds. While we do not have specific results to mention here, below we will discuss some methods to build symplectic cobordisms (that might be a step towards building symplectic caps) and build symplectic caps for a very restricted class of iterated planar contact $5$--manifolds. We note that recently, Conway and the second author \cite{ConwayEtnyre20} and, independently, Lazarev \cite{Lazarev20b} showed that strong symplectic caps exist for arbitrary contact manifolds, but there is little one can guarantee about the topology of these caps, and hence it does not seem likely they can be used to study iterated planar contact manifolds as the caps in Item~\eqref{3} were used to restrict planar contact manifolds. 

We have seen that there are many contact manifolds that are iterated planar, but we end this section by pointing out that not all higher dimensional contact manifolds are iterated planar.  In \cite{MassotNiederkrugerWendl12}, Massot, Niederkr\"uger, and Wendl gave examples of symplectic manifolds with disconnected weakly convex boundary. It follows that none of the boundary components of these examples are  iterated planar by  Theorem~\ref{obstr1}.

\subsection{Projective open books and adapted contact manifolds}\label{sec: pobs}
There is another possible generalization of planar open books to higher dimensions. The idea for this comes from noticing that $S^2$ is also $\C P^1$ and the page of a planar open book naturally embeds in $\C P^1$. 

\begin{definition} We say that an open book for a $(2n+1)$--manifold is {\em projective} if its Weinstein page embeds as a convex domain in $\C P^n\#_k \overline{\C P^n}$ for some $k$. A contact manifold supported by a projective open book is called projective.
\end{definition} 

A more natural term for such open books would be ``rational", but as this term is already used for something else when discussing open books \cite{BakerEtnyreVanHorn-Morris12} we have opted for ``projective". 

Our first observation is that in dimension 5, iterated planar open books are also projective, and hence  iterated planar contact $5$--manifolds are projective.
\begin{theorem}\label{ipisp}
Any iterated planar contact $5$--manifold is projective.
\end{theorem}
Combining Theorem~\ref{ipisp} with the results in Section~\ref{sec: itpobs}, we conclude that there are many projective contact $5$--manifolds, which leads us to ask:
\begin{question}
Are iterated planar contact manifolds of dimension greater than five also projective?
\end{question}

Our main result about projective open books is that Eliashberg's ``capping procedure" \cite{Eliashberg04} can partially be carried out in this context (see Section~\ref{capingproceedure}) and in some cases fully carried out.
\begin{theorem}\label{projcap}
If $(M^5,\xi)$ is supported by an open book whose page embeds as a convex domain in $\C P^2\#_n \overline{\C P}^2$ for $n\leq 4$, then $(M,\xi)$ has a symplectic cap that contains an embedded $\C P^2\#_n \overline{\C P}^2$.
\end{theorem}
\begin{remark}
It is conjectured that any symplectomorphism of $\C P^2\#_n \overline{\C P}^2$ is symplectically isotopic to a composition of Dehn twists about Lagrangian spheres \cite{LiWuPre}. If this were known to be true then the above theorem would be true without the restriction on $n$. 
\end{remark}

\begin{example}
It is well known that the anti-diagonal is a Lagrangian sphere  in $S^2\times S^2$ with the symplectic form $\omega\oplus \omega$ and thus there is one in $\C P^2\#_2\overline{\C P}^2\cong (S^2\times S^2)\#\overline{\C P}^2$ with its blown up symplectic form. As Lagrangians have standard neighborhoods we see that the unit cotangent bundle $DT^*S^2$ of $S^2$ embeds in  $\C P^2\#_2\overline{\C P}^2$. Thus any contact structure on a $5$--manifold supported by an open book with page $DT^*S^2$ is a projective open book satisfying the hypothesis of this theorem.  We recall in Example~\ref{ex:ust} below that $S^5$ has infinitely many distinct contact structures that are supported by an open book with page $DT^*S^2$. So all of these contact structures have a cap containing $\C P^2\#_2\overline{\C P}^2$.
\end{example}

\begin{example}
Let $W_d^4$ be the Stein domain which is the complement of a neighborhood of a symplectic hypersurface $\Sigma_d$ of degree $d > 1$ 
in $\mathbb{CP}^2$ and  let $\tau_d$ denote the fibered Dehn twist on $W^4_d$ along its boundary $\del W^4_d$. The lens space $L_d$ given by $S^5/\mathbb{Z}_d$,  admits an open book with page $W_d$ and monodromy $\tau_d$, see  \cite[page 423]{ChiangDingvanKoert14}. Moreover the supported contact structure $\xi_d$ is simply the one induced from $S^5$ by taking the quotient by the group action. So from Theorem~\ref{projcap} we see that $(L_d,\xi_d)$ has a symplectic cap that contains a copy of $\C P^2$. We would like to point out $(L_2,\xi_2)$  will appear again in the proof of  Proposition~\ref{fillablebutnoteStein}. 
\end{example}

\subsection{Building symplectic cobordisms using open books}
Here we discuss two methods to build symplectic cobordisms based on information about open book decompositions. 

\begin{theorem} \label{thm: cob} 
Let $(M,\xi)$ be a $(2n+1)$--dimensional closed contact manifold supported by an open book decomposition with page $F$, binding $B=\partial F$, and monodromy $\phi$. Suppose that $X$ is an exact symplectic cobordism such that $\partial X=-B \cup B'$. Then there exists a strong symplectic cobordism $$W=M \times [0, 1] \bigcup_{B \times D^2 \times \{1\}=\partial_{-}X \times D^2} X\times D^2$$ (after rounding corners)  such that $\partial W=\partial_{-}W \cup \partial_{+}W$ with $\partial_{-}W=-M$ and $\partial_{+}W=M'$ where $(M', \xi')$ is supported by an open book with page $F \cup X$, binding $B'$, and monodromy $\phi'$ which is identity on $X$ and agrees with $\phi$ on $F$.
\end{theorem}

During the writing of this paper the first author and Agustin Moreno independently constructed a cobordism in \cite{Acu17preprint} using cobordisms of the binding as in Theorem~\ref{thm: cob}  and, though not stated in these terms, their arguments can recover Theorem~\ref{thm: cob} (and more). We also note that this theorem can be thought of as a higher dimensional analog of a cobordism constructed by Lisi, Van Horn-Morris, and Wendl in \cite{LisiVanHornMorrisWendl18} for ``spinal open books". In addition, if $X$ in the theorem is a Weinstein cobordism then this theorem also follows from the work of van Koert in Section~4.1 and~4.3 of \cite{vanKoert17}

We hope that Theorem~\ref{thm: cob} might be useful in studying symplectic cobordisms between contact manifolds. The next theorem does not give a symplectic cobordism between contact manifolds but it rather gives a symplectic cobordism from a contact manifold to a symplectic fibration. This will be used in conjunction with the information about the symplectomorphism group of some symplectic manifolds to prove Theorem~\ref{projcap}.

\begin{theorem}[D\"orner, Geiges and Zehmisch \cite{DornerGeigesZehmisch14}] \label{thm: cob1} 
Suppose that $(M,\xi)$ is a closed contact manifold supported by an open book decomposition with page $F$, binding $B=\partial F$, and monodromy $\phi$, and let $X$ be a symplectic cap for $B$. Then there exists a symplectic cobordism $$W=M \times [0, 1] \bigcup_{B\times D^2 \times \{1\}=\partial_{-}X \times D^2} X\times D^2$$  (after rounding corners) such that $\partial W=\partial_{-}W \cup \partial_{+}W$ with $\partial_{-}W=-M$ and $\partial_{+}W=M'$ where $M'$ fibers over the circle with symplectic fibers. The fibers of $M'$ are $F \cup X$ and the monodromy of the bundle is the identity on $X$ and agrees with $\phi$ on $F$.
\end{theorem}

\subsection*{Acknowledgements} We thank anonymous referees who provided valuable comments on earlier versions of this paper and caught an error in the original proof of Lemma~5.2. The second author was partially supported by the Elaine M.~Hubbard Distinguished Faculty Award and NSF grants DMS-1906414 and DMS-2203312.

\section{Iterated Planar Contact Manifolds}\label{ipcm}
In this section, we introduce iterated planar Lefschetz fibrations, iterated planar open books, and iterated planar contact manifolds, after recalling the necessary background on open books and Lefschetz fibrations.

\subsection{Open book decompositions and Lefschetz fibrations}
We begin by recalling two notions of open book decompositions. 

\begin{definition} \label{definition:aobd}
An \textit{abstract open book decomposition} is a pair $(F, \phi)$, where $F$ is a compact manifold with boundary, called the \textit{page} and $\phi: F \rightarrow F$ is a diffeomorphism which restricts to the identity on  $\partial F$, called the \emph{monodromy}.
\end{definition}

\begin{definition} 
An \emph{open book decomposition} of a closed oriented manifold $M$ is a pair $(B,\pi)$, where $B$ is a codimension $2$ submanifold of $M$ with trivial normal bundle, called the \emph{binding} of the open book and $\pi: M \setminus B \rightarrow S^{1}$ is a fiber bundle such that $\pi$ agrees with the angular coordinate $\theta$ on the normal disk $D^2$ when restricted to a neighborhood $B \times D^2$ of $B$. 
\end{definition}

For any $\theta \in S^1$, the closure $F_{\theta}:= \overline{\pi^{-1}(\theta)}$ is a manifold with boundary $\partial F_{\theta}= B$ called the \emph{page} of the open book. The holonomy of the fiber bundle $\pi$ determines a conjugacy class in the orientation-preserving diffeomorphism group of a page $F_{\theta}$ fixing its boundary, i.e. in $\mbox{Diff}^{+}(F_{\theta}, \partial F_{\theta})$ which we call the {\em{monodromy}}. Note that the pages are naturally co-oriented by the canonical orientation of $S^1$ and hence, since $M$ is oriented,  are naturally oriented. Also, the binding inherits an orientation from the open book decomposition of $M$. We assume that the given orientation on $B$ coincides with the boundary orientation induced by the pages. It is well-known that an abstract open book decomposition gives a closed  manifold with an open book decomposition and vice-versa. 

\begin{definition}
 A contact structure $\xi$ on a compact and oriented manifold $M$ is said to be \emph{supported by an open book} $(B,\pi)$ of $M$ if it is the kernel of a contact form $\lambda$ satisfying the following:
\begin{enumerate}
\item $\lambda$ is a positive contact form on the binding and
\item $d\lambda$ is positively symplectic on each fiber of $\pi$.
\end{enumerate}

If these two conditions hold, then the open book $(B,\pi)$ is called a \emph{supporting open book} for the contact manifold $(M, \xi)$ and the contact structure $\xi$, as well as the contact form $\lambda$, is said to be \emph{adapted} to the open book $(B,\pi)$. 
\end{definition}

\begin{definition} \label{def: Liouv}
A \textit{Liouville domain} is a compact manifold $W$ with boundary, together with a $1$--form $\lambda$ such that  $\omega = d \lambda$ is symplectic and the Liouville vector field $Z_\lambda$ on $W$ defined by $\omega (Z_\lambda, \cdot) = \lambda$ points transversely outward along $\partial W$. Moreover, if $\varphi : W \to \R$ is a (generalized) Morse function for which $Z_\lambda$ is gradient-like and $\partial W$ is a regular level set of $\varphi$, then the quadruple $(W, \omega, Z_\lambda, \varphi)$  is called a \textit{Weinstein domain}. 
\end{definition}

Note that the primitive $\lambda$ in Definition~\ref{def: Liouv} is called  a \textit{Liouville  form} and $\ker \lambda$ is a contact structure on $\partial W$, by definition. We also say that $\partial  W$ is {\em convex. }

Let $(W, \beta)$ be a Liouville domain and let $\phi : W \to W$ be an exact symplectomorphism which is the identity near $\partial W$. Exactness of $\phi$ means that $\phi^* \beta - \beta =
d h$ for some positive function $h :  W \to \mathbb{R}$, which is
constant on each component of $\partial W$. 

\begin{remark} If  $\dim
W= 2$, then any \emph{diffeomorphism} of $W$ can be isotoped rel. boundary to a
symplectomorphism of $(W, d\beta)$ \cite[Lemma 7.3.2]{Geiges08}. When
$\dim W \geq  2$, exactness of a \emph{symplectomorphism} can be
achieved by an isotopy rel. boundary\cite[Lemma 7.3.4]{Geiges08}.
\end{remark} 

Note that $\ker (\beta|_{\partial W}) $ is a contact
structure on $\partial W$. Then $\alpha:=\beta + d\varphi$, where
$(r, \varphi)$ are polar coordinates on $\mathbb{D}^2$, descends to
a contact form on the (generalized) mapping torus $(W \times [0,1])/ \sim$, which can be extended
over $\partial W \times \mathbb{D}^2$ as $h_1(r)\beta+ h_2(r) d
\varphi$, for some suitable functions $h_1$ and $h_2$ (see
\cite[Section 7.3]{Geiges08}). Let $\lambda$ denote the resulting contact  form  on the manifold $M$ described by the abstract open book $(W, \phi)$. 

\begin{proposition} [Giroux 2002, \cite{Giroux02}]  
The contact structure $\xi=\ker \lambda$ described above on $M$ is adapted to the abstract open book whose page is the Liouville domain $(W, \beta)$ and whose monodromy is the exact symplectomorphism $\phi$. 
\end{proposition}

Conversely,

\begin{theorem}[Giroux 2002, \cite{Giroux02}] \label{gir} 
Every closed contact manifold admits an adapted open book with Weinstein pages.
\end{theorem} 

\begin{definition} \label{def: lef} 
Let $E$ be a compact $2n$--dimensional manifold with corners whose boundary is the union of two
faces, namely the horizontal boundary $\partial_h E$ and the vertical boundary $\partial_v E$, meeting in a codimension 2 corner. Let $\omega=d \lambda$ be an exact symplectic form on $E$ such that both faces of the boundary are convex. Let $f: E \to D^2$ be a proper smooth map with finitely many critical points  $Crit(f) \subset E$ and let  $F$ denote a regular fiber. We say that  $f$ is an \textit{exact symplectic Lefschetz fibration} on $(E, \omega= d \lambda)$ if it satisfies the following properties: 

(1) \textit{Conditions on the boundary}:  We require that $\partial_v E= f^{-1}(\partial D^2)$ and $f|_{\partial_v E}: \partial_v 
E \to \partial D^2$ is a surjective smooth fiber bundle. Moreover, there is a neighborhood $N$ of $\partial_h E$ such that $f|_N : N \to  D^2$ is a product fibration $D^2 \times nhd (\partial F)$, where $nhd (\partial F)$ is a collar neighborhood of $\partial F$ in  a regular fiber $F$,  and  
$\omega|_N$ and $\lambda|_N$ are sums of pullbacks of the corresponding form from the first and second factor of this product. In particular, $f|_{\partial_h E}: \partial_h E \to D^2$ is a surjective smooth fiber bundle, where $\partial_h E= \bigcup_{z \in D^2} \partial (f^{-1}(z))$. 

(2) \textit{Conditions on the fibers}:  Let $E_z$ denote the fiber $f^{-1}(z)$ for any $z \in D^2$. We require that the restriction of $\omega$ to $E_z \setminus Crit(f)$ is symplectic. In particular,  the boundary of each fiber is convex and each regular fiber of $f$ is an exact symplectic manifold with convex boundary. We also require that there is at most one critical point at each fiber. 

(3)  \textit{Conditions on the critical points}: For any $p \in Crit(f)$,  there are orientation preserving local complex coordinates about $p$ on $E$ and $f(p)$ on $D^2$ such that, with
respect to these coordinates, $f$ is given by the  map $$f(z_1, . . . , z_n)= z_1^2 + \cdots + z_n^2.$$  
\end{definition}

It follows by  Definition~\ref{def: lef} that the critical points of $f$ are isolated and belong to the interior of $E$. We also allow $Crit(f)$ to be the empty set. More importantly, the corners of $E$ can be rounded off to obtain a  Liouville domain $(W, \lambda)$ so that one can speak about an exact symplectic Lefschetz fibration $f: (W, \lambda)  \to D^2$ such that each regular fiber is again a Liouville domain equipped with the restriction of $\lambda$. 

According to a recent result of Giroux and Pardon \cite{GirouxPardon17}, any Stein (respectively Weinstein) domain admits a Stein (respectively Weinstein) Lefschetz fibration over $D^2$, which by definition implies that the regular fiber of the fibration  is Stein (respectively Weinstein).  So, we can consider these type of fibrations as special cases of the exact symplectic fibrations we defined on Liouville domains---which might be called Liouville Lefschetz fibrations in the same vein. 

\begin{construction}\label{theconstruction}
We make a few observations about constructing open books and Lefschetz fibrations that will be used repeatedly below. The first simple observation is that if $f: (W, \lambda)  \to D^2$ is an exact symplectic Lefschetz fibration,  then $\partial W$ is convex and hence has an induced contact structure $\xi$. Moreover, if we take polar coordinates on $D^2$ and let $B=f^{-1}(0)\cap \partial W$, then $\partial W \setminus B$ is fibered by composing $f$ with projection to the $\theta$ coordinate. One may easily check that this open book supports $\xi$. In addition, \cite[Lemma 4.4]{vanKoert17} implies that  if $\Lambda$ is an exact Lagrangian sphere in a fiber of $f$ over a point in $\partial D^2$, then there is an associated Legendrian sphere in $\partial W$ in a neighborhood of that fiber, obtained by perturbing the Lagrangian one. By attaching a Weinstein handle to $W$ along $\Lambda$, we get a new symplectic manifold $W'$ with an exact symplectic Lefschetz fibration that has one additional vanishing cycle given by $\Lambda$, \cite[Section~6.2]{GirouxPardon17}. In other words,  given an exact symplectic $(2n-2)$--manifold $X$, a general exact symplectic Lefschetz fibration in dimension $2n$ can be built by successively attaching Weinstein $n$--handles to the trivial exact symplectic Lefschetz fibration $X \times D^2$ along Lagrangian spheres embedded in distinct fibers above $\partial D^2$. 
\end{construction}

\subsection{Iterated planar open books}\label{sec:IPOB}

Once an exact symplectic Lefschetz fibration on a Liouville domain is given,  one can further consider an exact symplectic Lefschetz fibration on the codimension two Liouville fiber, and iterate this dimension reduction until the fiber is $4$--dimensional.  This leads to the following definition. 

\begin{definition}
An \textit{iterated planar Lefschetz fibration} on a $2n$--dimensional Liouville domain $(W,  \lambda)$ is a sequence of exact symplectic Lefschetz fibrations $$f_i:(W_i,  \lambda_i) \to D^2$$ for $i=n,n-1, \ldots, 2$, with the following properties: 
\begin{enumerate}
\item $(W, \lambda) = (W_n, \lambda_n)$ and $f=f_n$
\item $(W_{i-1},  \lambda_{i-1})$ is  a regular fiber of $f_{i}$, for $i=n,n-1, \ldots, 3$ 
\item $f_2: (W_2,  \lambda_2) \to D^2$ is planar, i.e. the regular fiber of $f_2$ is a genus zero surface with nonempty boundary.
\end{enumerate}
Notice that when thinking of $(W_{i-1},  \lambda_{i-1})$ as a fiber of $f_i$ then we are rounding corners as discussed in Definition~\ref{def: lef} so that we never have corners of codimension greater than two. 
\end{definition}

Although we formulated the definition of an iterated planar Lefschetz fibration by reducing the dimension at each step, it is also possible to start with the lowest possible dimension and build up higher-dimensional fibrations as follows. Suppose that $W_2$ is a smooth $4$--manifold with nonempty boundary which admits a {\em smooth} planar Lefschetz fibration over $D^2$. Then by standard methods one can equip $W$ with an exact symplectic form $\omega_2=d\lambda_2$ such that the smooth planar Lefschetz fibration turns into an exact symplectic Lefschetz fibration $f_2:(W_2,  \lambda_2) \to D^2$ with planar fibers. Similarly, using Construction~\ref{theconstruction} one can construct a Liouville domain $(W_3,  \lambda_3)$ of dimension $6$ which admits an exact symplectic Lefschetz fibration $f_3 : (W_3,  \lambda_3) \to  D^2$, whose fiber is  $(W_2, \lambda_2)$. This process, of course, can be iterated as many times as desired.  Note that an iterated planar Lefschetz fibration on a $4$--dimensional Liouville domain is nothing but a planar Lefschetz fibration.

\begin{remark} 
It is not true that every $4$--manifold with nonempty boundary admits a planar Lefschetz fibration over $D^2$. For example, suppose that  $T^2 \times D^2$ admits such a fibration. Then there must be a planar Stein fillable contact structure on $\partial (T^2 \times D^2) = T^3 , $ but the unique Stein fillable contact structure on $T^3$ \cite{Eliashberg96} is known to be not planar \cite{Etnyre04b} (see
Item (2) in the Introduction).
\end{remark}

\begin{definition}
An open book decomposition whose (Weinstein) page admits an iterated planar Lefschetz fibration is called an \textit{iterated planar open book}. An \textit{iterated planar contact manifold}  is a closed contact manifold of dimension at least 5, which is supported by an iterated planar open book decomposition.
\end{definition}

\section{General results about iterated planar contact manifolds}\label{IPgeneral}

A few examples of iterated planar contact manifolds  were implicit in the first author's work  \cite{Acu17preprint}. 
\begin{example}\label{ex: simple}
The simplest examples of iterated planar contact manifolds are  
\begin{enumerate} 
\item the standard contact sphere $(S^{2n+1}, \xi_{std})$,  
\item the unit cotangent bundle $(ST^*S^n, \eta)$ of $S^n$ equipped with its canonical contact structure $\eta$, and 
\item  the convex boundary $(S^n\times S^{n+1}, \xi)$ of the Weinstein domain $DT^*S^n\times D^2$, where $DT^*S^n$ denotes the unit disk cotangent bundle of $S^n$. 
\end{enumerate}
To see this, we will first show that $DT^*S^n$ has the structure of an iterated planar Lefschetz fibration, which immediately implies that $(ST^*S^n, \eta)$ is iterated planar. We will  make use of Construction~\ref{theconstruction} repeatedly in the discussion below. Let $L$ be a linear Lagrangian disk in  the standard symplectic ball $B^{2n}$ so that $\partial L = \Lambda$ is a Legendrian sphere in the standard contact $S^{2n-1}$. If we attach a Weinstein $n$--handle to $B^{2n}$ along $\Lambda$ it is easy to see that we obtain  $F=DT^*S^n$ that contains a Lagrangian $S^n$, that we denote $L'$. Now $D^2\times F$ is a subcritical Weinstein manifold and in its boundary we can take $L'$ to be a Legendrian sphere. Attaching a $(2n+2)$--dimensional $(n+1)$--handle to $D^2\times F$ along $L'$ will result in $B^{2n+2}$ with a symplectic form deformation equivalent to the standard symplectic form. Thus $B^{2n+2}$ has the structure of a Lefschetz fibration with fiber $F$ and one vanishing cycle given by $L'$. But now attaching another $(2n+2)$--dimensional handle to $B^{2n+2}$ along a copy of $L'$ will result in $DT^*S^{n+1}$ with a Lefschetz fibration with fiber $DT^*S^n$ and two vanishing cycles. That is for all $n \geq 1$, $DT^*S^{n+1}$ has a Lefschetz fibration with fiber $DT^*S^n$ and in particular for $n=1$, the fiber  $DT^*S^1=S^1\times [-1,1]$ is planar. Therefore, we conclude that $DT^*S^{n}$ (and hence $B^{2n+2}$) has the structure of an iterated planar Lefschetz fibration. 

Now the iterated planar Lefschetz fibration structure on $DT^*S^n$ yields an iterated planar Lefschetz fibration structure on $DT^*S^n  \times D^2$ and hence induces an iterated planar open book supporting its contact boundary  $(S^n\times S^{n+1}, \xi)$. 

The iterated planar Lefschetz fibration structure on $B^{2n+2}$ constructed above induces an iterated planar open book supporting $\xi_{std}$ on $S^{2n+1}$. The "iterated" pages are $DT^*S^n$ and the "iterated" monodromy is a Dehn twist about $L'$. 
\end{example}

 \begin{example}\label{ex: brieskorn} 
The Brieskorn manifold $\Sigma^{2n-1} (k, 2, \ldots, 2) \subset  \C^{n+1} $  is defined as the intersection of the sphere $S^{2n+1}$ 
with the zero set of the polynomial $ z_0^k + z_1^2 + \cdots + z_n^2$. Viewed as a singularity link, $\Sigma^{2n-1}(k, 2, \ldots, 2 )$ carries a  canonical contact structure $\eta_k$. The contact manifold $(\Sigma^{2n-1}(k, 2, \ldots, 2), \eta_k)$  is supported by an open book with page $DT^*S^{n-1}$, and monodromy the $k$-fold right-handed Dehn twist along the zero section  \cite{vanKoertNiederkruger05}.  It follows that $(\Sigma^{2n-1}(k, 2, \ldots, 2), \eta_k)$ is iterated planar by Example~\ref{ex: simple}.  Note that the corresponding Milnor fiber admits a Lefschetz fibration with fiber $DT^*S^{n-1}$, and hence it admits an iterated planar Lefschetz fibration structure.   \end{example} 

More generally,  we can show that any finitely presented group is the fundamental group of some iterated planar contact manifold. 

\begin{proof} [Proof of Theorem~\ref{IPgroups}]
First we prove the case $n=2$. For each finitely presented group $G$, there is a planar  Lefschetz fibration $W^4 \to D^2$ with $\pi_1(W^4) \cong G$ (see \cite[Proposition 6.1]{GhigginiGollaPlamenevskaya20}). Let $M^5$ be the contact $5$--manifold supported by the open book with page $W^4$ and monodromy the identity map. Then, notice that $M^5$ is simply the double of $W^4  \times [0,1]$. (Note that a neighborhood of a page of the open book is $W^4  \times[0,1]$ and its complement is also diffeomorphic to $W^4  \times [0,1]$. Since the monodromy of the open book is the identity, these two pieces are glued by the identity map, resulting the  double of $W^4  \times [0,1]$.) That is, it is obtained by attaching $3$--, $4$-- and $5$--handles to $W^4  \times [0,1]$. In particular, 
$$\pi_1(M^5) \cong \pi_1 (W^4) \cong G.$$ 
Therefore, $M^5$ is  an iterated planar contact $5$--manifold whose fundamental group is $G$.  Now, consider the contact $7$--manifold $M^7$ supported by the open book with  page $W^4 \times D^2$ (which is a Stein filling of $M^5$) and monodromy the identity map. The discussion above shows that $\pi_1(M^7) \cong \pi_1(W^4 \times D^2)  \cong \pi_1 (W^4) \cong G$ and $M^7$ is iterated planar by definition. It is clear that this process can be iterated for all $n \geq2$, to construct an iterated planar contact $(2n+1)$--manifold whose fundamental group is $G$.  
\end{proof}
\begin{remark}
We note that in the proof above, one does not need to use the identity monodromy to construct a contact manifold with a given fundamental group. If $W^4$ has a non-trivial symplectomorphism $\phi$, then the $5$--manifold $M^5$ with page $W^4$ and monodromy $\phi$ will still have fundamental group $G$ since one would be building $M$ from $W^4 \times [0,1]$ by attaching $3$--, $4$--, and $5$--handles (although  the attaching maps are changed by $\phi$). 
\end{remark}

Next we will show that the connected sum preserves iterated planarity, as a corollary of Lemma~\ref{lem: LFconnect}.

\begin{lemma}\label{lem: LFconnect} Suppose that $f_i: (W_i, \lambda_i)  \to D^2$ is an exact Lefschetz fibration with regular fiber $F_i$, for $i=0,1$. Then a $1$--handle can be attached to $W_0\cup W_1$ to obtain a Weinstein manifold $W$ which admits an exact Lefschetz fibration with regular fiber $F_0\cup F_1$ with a $1$--handle attached. 
\end{lemma}

\begin{proof} The boundary of $W_i$ naturally splits into two pieces:  the vertical boundary  $f_i^{-1}(\partial D^2)$ and the horizontal boundary $(\partial F_i)\times D^2$, where $F_i$ denotes a fiber of $f_i$. Note that $(\partial F_i)\times D^2$ has a natural contact structure with contact form $\alpha_i+ x\,dy-y\,dx$, where $\alpha_i$ is a contact form for $\partial F_i$ and $(x,y)$ are coordinates on $D^2$. Supposing that $B_i$  is a Darboux ball in $\partial F_i$, we can attach a $2n$--dimensional $1$--handle to $W_0 \cup W_1$ along $(B_0\times D^2)\cup (B_1\times D^2)$. More specifically, if we take a $(2n-2)$--dimensional Weinstein $1$--handle $D^1\times D^{2n-3}$ that is attached along $(\partial D^1)\times D^{2n-3}$, then $D^1\times D^{2n-3}\times D^2$ is a $2n$--dimensional $1$--handle that can be thought of as a $D^2$'s worth of $(2n-2)$--dimensional Weinstein $1$--handles. Therefore, a $2n$--dimensional $1$--handle can be attached to $[(\partial F_1)\cup (\partial F_2)]\times D^2$  by attaching $(\partial D^1)\times D^{2n-3}\times \{p\}$ to $[(\partial F_1)\cup (\partial F_2)]\times \{p\}$ for each $p\in D^2$. That is for each $p\in D^2$, we attach a $(2n-2)$--dimensional $1$--handle to $f_1^{-1}(p)\cup f^{-1}_2(p)$. 
We conclude that the Weinstein manifold $W = W_0\cup W_1\cup (\text{$1$--handle})$ admits an exact  Lefschetz fibration with fiber $F_0\cup F_1\cup (\text{$1$--handle})$. 
\end{proof} 

\begin{proof}[Proof of Theorem~\ref{IPconnect}]
Suppose that $(M_i,\xi_i)$ is a contact manifold supported by an iterated planar open book $(W_i, \phi_i)$, for $i=0,1$. We claim that we can attach a $1$--handle to $W_0\cup W_1$ to obtain a Weinstein manifold $W$ that has the structure of an iterated planar Lefschetz fibration. If $\phi:W\to W$ denotes the symplectomorphism that restricts to $\phi_i$ on $W_i$ and is the identity on the $1$--handle, then one may readily check using \cite[Proposition~4.2]{vanKoert17} that  the open book $(W,\phi)$ supports $(M_0\#M_1,\xi_0\# \xi_1)$ --- which gives a proof of the desired result assuming the claim. 

To prove our claim, we first observe that $W_0$ and $W_1$ both have the structure of an iterated planar Lefschetz fibration,  by definition. Then a $1$--handle can be attached to $W_0\cup W_1$ to obtain a Weinstein manifold $W$ which admits an exact Lefschetz fibration, whose fibers are obtained from the fibers of the Lefschetz fibrations on $W_0$ and $W_1$ by Lemma~\ref{lem: LFconnect}.   Iterating this construction until we get to the $4$--dimensional Lefschetz fibration case, we see that each fiber is simply obtained from the  $2$--dimensional fibers for $W_0$ and $W_1$ together with a $1$--handle attached to connect them. But since these $2$--dimensional fibers for $W_0$ and $W_1$ are both planar, so is the $2$--dimensional fiber for $W$. Therefore, the exact Lefschetz fibration on $W$ is indeed iterated planar. 
\end{proof} 

\begin{remark}
We will see an alternate proof of Theorem~\ref{IPconnect} in the $5$--dimensional case below. 
\end{remark}

\section{Iterated planar contact manifolds in dimension 5}\label{IP5D}

We begin this section with a simple observation. 
\begin{lemma} \label{lem: bind}
A contact $5$--manifold is iterated planar if and only if it admits an adapted open book whose binding is planar. \end{lemma}

\begin{proof} 
If a contact $5$--manifold  is iterated planar,  then by definition, it admits a supporting  open book whose Weinstein page admits an exact symplectic Lefschetz fibration over $D^2$  with planar fibers, and hence  the binding of this open book is a planar contact $3$--manifold. Conversely, if a contact $5$--manifold admits an adapted open book with planar binding, then since the Weinstein page of this open book is a strong filling of its binding and any strong symplectic filling of a planar contact $3$--manifold admits an exact symplectic planar Lefschetz fibration over $D^2$ by Wendl's work \cite{Wendl10}, we conclude that the contact $5$--manifold at hand is iterated planar. 
\end{proof}
Using this lemma, we can give an alternate proof of the fact that the connected sums preserve iterated planarity in dimension 5. 
\begin{proof}[Alternate proof of Theorem~\ref{IPconnect} in dimension 5]
If $(M^5_i, \xi_i)$ is iterated planar, then by definition it admits an adapted open book with Weinstein page $W_i^4$ so that the binding $\partial W_i^4$ is planar. Note that the contact connected sum $(M^5_1 \#M^5_2, \xi_1 \# \xi_2)$ is supported by the open book whose page is obtained as the boundary connected sum $W^4_1 \natural W^4_2$. Hence the binding of this open book is the contact connected sum $\partial W_1^4 \# \partial W_2^4$. The connected sum of planar contact manifolds is planar, see \cite{Torisu00}. Therefore,  the binding of the open book supporting $(M^5_1 \#M^5_2, \xi_1 \# \xi_2)$ is planar and thus Lemma~\ref{lem: bind} implies that the open book is iterated planar. 
\end{proof}

We now consider contact $5$--manifolds with subcritical Stein fillings and prove Theorem~\ref{IPsubcrit1}. 

\begin{proof}[Proof of Theorem~\ref{IPsubcrit1}]
We begin with Item~(1). We are considering the contact structure on $S^1\times S^4$ filled by he Stein manifold $S^1\times D^5$ obtained by attaching a Stein $1$--handle to the standard symplectic $B^6$. This induces an open book supporting the contact structure with pages $S^1\times D^3$, which admits a Lefschetz fibration with fiber $S^1\times D^1$. Thus the open book, and hence the aforementioned contact structure on $S^1\times S^4$ is iterated planar.

For Item~(2) we recall that subcritical Stein fillings of such manifolds have been classified by Ding, Geiges, and Zhang \cite{DingGeigesZhang18} as follows. We first define $L_n$ to be the lens space $L(n,1)$ with an open $3$--ball removed. Suppose that $(M^5,\xi)$ is subcritically Stein fillable contact $5$--manifold whose fundamental group is $\Z/n\Z$. Let $r$ be the rank of $H_2(M;  \Z)$. If $n$ is odd then $M$ is diffeomorphic to 
\[
\partial (L_n\times D^3) \#_r (S^2\times S^3) \text{ or } \partial (L_n\times D^3) \# (S^2\widetilde\times S^3) \#_{r-1} (S^2\times S^3),
\]
depending on whether $M$ is spin or not, where $S^2\widetilde\times S^3$ is the nontrivial $S^3$-bundle over $S^2$. If $n$ is even and $M$ is spin then it is diffeomorphic to
\[
\partial (L_n\times D^3) \#_r (S^2\times S^3),
\]
and if $M$ is not spin then it is diffeomorphic to 
\[
\partial (L_n\widetilde \times D^3) \#_r (S^2\times S^3) \text{ or } \partial (L_n\times D^3) \# (S^2\widetilde\times S^3) \#_{r-1} (S^2\times S^3), 
\] where $L_n\widetilde \times D^3$ is the nontrivial $D^3$-bundle over $L_n$. Moreover, in each homotopy class of almost contact structures, there is a unique subcritically Stein fillable contact structure. 

We claim that homotopy classes of almost contact structures on a connected sum are in one-to-one correspondence with the product of homotopy classes on the summands. To see this recall that homotopy classes of almost contact structures on a $5$--manifold correspond to homotopy classes of sections of the $SO(5)/U(2)$-bundle associated to the tangent bundle of the manifold, and note that $SO(5)/U(2)$ is homotopy equivalent to $\C P^3$ \cite{Geiges08}. Thus,  if we consider the neighborhood of a $4$--sphere in a $5$--manifold it will admit a unique almost contact structure (up to homotopy). So given an almost contact structure on $M_1\# M_2$ it will induce the unique such structure on a neighborhood of the connect summing sphere, and hence we have almost contact structures induced on $M_1$ and $M_2$. It is clear that given such structures on $M_1$ and $M_2$ that we can get one on $M_1\# M_2$ and these constructions are inverses of each other. 

Given Theorem~\ref{IPconnect} concerning connected sums of iterated planar contact structures (and the results above), it will suffice to prove that each homotopy class of almost contact structures on the summands above, is realized by a subcritically Stein fillable contact structure that is supported by an iterated  planar open book. 

We begin with $S^2\times S^3$ and first recall that the possible framings on an $S^1$ in a $5$--manifold are given by $\pi_1(SO(4))=\Z/2\Z$. Thus if we consider a $4$--dimensional handlebody $H$, then $H\times D^2$ is a 6-dimensional handlebody and the framings on the $2$--handles are reduced modulo 2.  Let  $W$ denote the Stein $4$--manifold obtained from $B^4$ by attaching a Stein $2$--handle along a  Legendrian unknot $L \subset (S^3,\xi_{std})$ with {\em odd} Thurston-Bennequin invariant. Then the Stein $6$--manifold $W\times D^2$ is diffeomorphic to $S^2\times D^4$ and thus its boundary is $S^2\times S^3$. As we can realize any possible Chern class on $W$ by an appropriate choice of $L$, we see that we can also realize any possible choice of Chern class on $W\times D^2$ and hence on $S^2\times S^3$. We recall that the Chern class does not uniquely specify the almost contact class of a contact structure. However, the ambiguity comes from the $2$-torsion in the second cohomology \cite{DingGeigesZhang18}. As there is no $2$--torsion in the homology of $S^2\times S^3$ this means that all almost contact structures are realized by this construction. Of course the binding of the open book coming from the product Lefschetz fibration $W\times D^2$ is simply $\partial W$. Since this is a lens space with a tight contact structure it is known to be planar \cite{Schoenenberger07}. Thus by Lemma~\ref{lem: bind}, all the subcritical Stein fillable contact structures on $S^2\times S^3$ are iterated planar. 

The same argument works for $S^2\widetilde\times S^3$ except that one uses a Legendrian unknot $L \subset (S^3,\xi_{std})$ with {\em even} Thurston-Bennequin invariant.

We now turn to $L_n\times D^3$.  As $L_n$ is a $3$--manifold built by a single 0-, $1$--, and $2$--handle, where the attaching region for the $2$--handle is a circle that is a $(n,1)$ curve on the solid torus $0$-handle $\cup$ $1$-handle, we see that $W_0$ in Figure~\ref{figsibcrit} is a handlebody picture for $L_n\times [0,1]$.  
\begin{figure}[htb]{\tiny
\begin{overpic}
{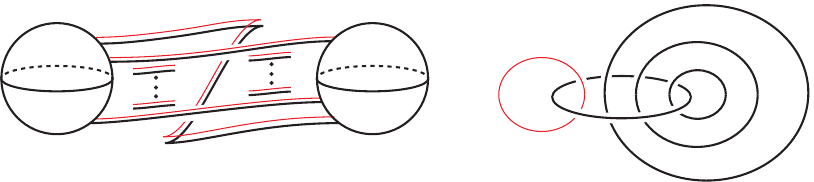}
\put(280, 20){$-1$}
\put(344,52){$-n$}
\put(354, 66){$\frac{-n}{n-1}$}
\put(368, 82){$-2n-2k$}
\put(240, 60){$0$}
\put(110, 15){$-n-2k$}
\end{overpic}}
\caption{The $4$--manifold $W_k$ shown on the left. The $2$--handle runs $n$ times over the $1$--handle. On the right, is the $3$--manifold boundary of $W_k$. The thin red curve is a regular fiber in the fibration and has contact framing $0$ with respect to the fibration framing. All surgery coefficients are with respect to the topological framings, not contact framings.}
\label{figsibcrit}
\end{figure}
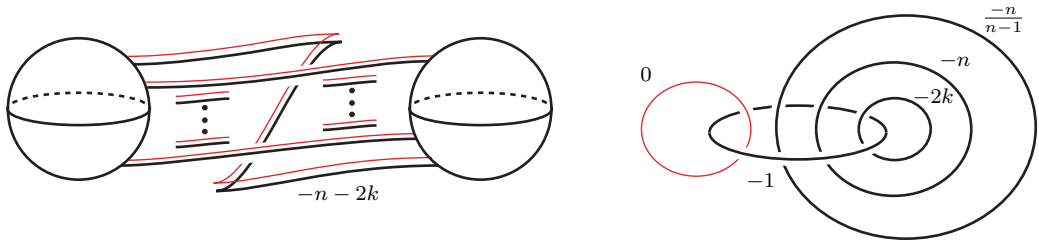

As noted above if we change the framing on the $2$--handle by some even number then we get a $4$--manifold $W_k$, that when crossed with $D^2$ is diffeomorphic to $L_n\times D^3$. Notice that Figure~\ref{figsibcrit} can be thought of as a Stein $S^1\times D^3$ with a Legendrian knot $L$ in its boundary, and $L$ has Thurston-Bennequin invariant $-n$. Thus stabilizing once and attaching a Stein handle to $L$ gives the manifold $W_1$. Similarly if we stabilize $L$, $2k+1$ times before attaching a handle we obtain a Stein realization of $W_{k+1}$. We notice that $\partial W_{k+1}$ is the Seifert fibered space shown on the right of Figure~\ref{figsibcrit}. In \cite{LiscaStipsicz07}, Lisca and Stipsicz showed that all contact structures with zero twisting are given by the surgery diagram in Figure~\ref{planarSFS}. 
\begin{figure}[htb]{\tiny
\begin{overpic}
{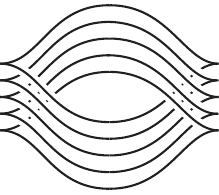}
\put(110, 26){$(+1)$}
\put(110, 35){$(+1)$}
\put(110, 44){$(-2k-2n)$}
\put(-24, 52){$(-n)$}
\put(110, 62){$\left(-\frac n{n-1}\right)$}
\end{overpic}}
\caption{A surgery diagram for $\partial W_{k}$. All surgery coefficients are with respect to the contact framing.}
\label{planarSFS}
\end{figure}
Zero twisting means that the regular fiber in the Seifert fibration can be realized by a Legendrian knot with contact framing agreeing with the framing coming form the fibration. The thin red curve in Figure~\ref{figsibcrit} is a regular fiber and we see that it has zero twisting. As noted  in \cite{LiscaStipsicz07} such a contact structure can easily be seen to be supported by a planar open book. Therefore, the boundaries of $W_{k+1}\times D^2$ have contact structures that are supported by iterated planar open books. 

So we get Stein manifolds $W_k\times D^2$ diffeomorphic to $L_n\times D^3$, and we can realize all possible Chern classes of $L_n\times D^3$. Thus as above, when $n$ is odd, so that there is no $2$--torsion in the homology of $L_n\times D^3$, we see that all subcritically fillable contact structures on $L_n\times D^3$ can be realized by iterated planar open books. 

When $n$ is even, there are at most two contact structures sharing the same Chern class, by a result  in \cite{DingGeigesZhang18}  which we recalled above, and since we can realize all possible Chern classes with our construction, we have shown at least half of the subcritically fillable contact structures are iterated planar. 

We now turn to Item~(3) in the statement of the theorem.  In the proof of this we will need the following result.

\begin{lemma}[Onaran 2021, \cite{Onaran20Pre}]]\label{lem: legplanar}
If $L$ is a topological link in $\#_k S^1\times S^2$ then there is a planar open book for $\#_k S^1\times S^2$ supporting the unique tight contact structure on this manifold such that $L$ can be realized as a Legendrian link on a page of the open book. 
\end{lemma}

Now if $(W,\omega)$ is a subcritical Stein filling of a given contact $5$--manifold $(M,\xi)$, then we know from Cieliebak \cite{Cieliebak02Pre} that $(W,\omega)$ is obtained from a Stein $4$--manifold $(X,\omega)$ by taking the product with $D^2$. The $4$--manifold $(X,\omega)$ is obtained from $(\natural_k S^1\times D^3,\omega_{std})$, where $\natural$ denotes the boundary sum, by attaching Weinstein $2$--handles along a Legendrian link $L$ in $(\#_k S^1\times S^2, \xi_{std})$. By Lemma~\ref{lem: legplanar}, there is some other Legendrian link $L'$ that is smoothly isotopic to $L$ and sits on a page of a planar open book for $(\#_k S^1\times S^2, \xi_{std})$. After stabilizing $L'$ more (and still calling it $L'$) we can assume the parity of the contact framing of $L$ agrees with the parity of the contact framing of $L'$ and that the components of the links have the same rotation number.  Let $(X',\omega')$ be the Stein domain obtained from $\natural_k S^1\times D^3$ by attaching Stein $2$--handles along the Legendrian link  $L'$. The boundary of $X'$ is a planar contact manifold, since the surgery link $L'$ sits on a page of a {\em planar} open book for $(\#_k S^1\times S^2, \xi_{std})$.  So $(X',\omega')$ admits a planar Lefschetz fibration. As discussed above, $X'\times D^2$ is diffeomorphic to $X\times D^2$ (because of our condition on the parity of contact framings) and the corresponding symplectic manifolds have homotopic almost complex structures (because of our condition on the rotation numbers). Thus $\partial (X' \times D^2)= \partial (X \times D^2)= \partial W = M$ has an iterated planar contact structure in the same homotopy class of almost contact structures as $\xi$. To prove the last part of Item~(3) we notice that by changing the rotation numbers on $L'$ we can achieve different Chern classes for $X'\times D^2$. Thus since $H^2(X'\times D^2 ; \Z)\cong H^2(\partial (X'\times D^2) ; \Z)$ (which follows by the Poincar\'{e}-Lefschetz duality, the long exact sequence for a pair, and the fact that $X'$ has the homotopy type of a $2$-complex) we see that we can realize infinitely many different homotopy classes of  almost contact structures by iterated planar contact structures. \end{proof}

In order to prove Theorem~\ref{felxibleIP} concerning iterated planar contact manifolds with flexible Stein fillings, we first need to consider open book decompositions for some contact structures on $S^5$. 

\begin{example}\label{ex:ust}
The Brieskorn manifold $(\Sigma^5 (k,2,2,2), \eta_k)$ is supported by the open book with page $DT^*S^2$,  and monodromy the  $k$--fold right-handed Dehn twist along the zero-section (see, Example~\ref{ex: brieskorn}). It follows that the iterated planar contact $5$--manifold $(\Sigma^5 (k,2,2,2), \eta_k)$ is Stein fillable and hence tight.  

For $k$ odd, $\Sigma^5 (k,2,2,2)$ is diffeomorphic to $S^5$, see \cite{Brieskorn66}.  With this identification,  we get an infinite set $\{ (S^5, \eta_k)\; | k \; \mbox{odd} \}$  of iterated contact manifolds. Moreover, Ustilovsky \cite{Ustilovsky99} showed that these contact manifolds are non-contactomorphic.  We also note that $(S^5,\eta_k)$ has a Weinstein filling $X_k$ which is obtained from $DT^*S^2\times D^2$ by attaching $k$ Weinstein $3$--handles along  copies of the zero section in $DT^*S^2$, and by Construction~\ref{theconstruction} we see that $X_k$ has the structure of an iterated planar Lefschetz fibration. Since the rank of $H_3(X_k, \Z)$ is equal to $k-1$,  none of these fillings can be subcritical for $k>1$. 

On the other hand, $\Sigma^5 (k,2,2,2)$ is diffeomorphic to $S^2 \times S^3$ for $k$ even.  Moreover, $(\Sigma^5 (2,2,2,2), \eta_2)$ is contactomorphic to $ST^* S^3 \cong S^2 \times S^3$ equipped  with its canonical contact structure $\xi_{can}$  \cite[Lemma~3.1]{KwonvanKoert16}.  We observe that, since the canonical contact structure on the unit cotangent bundle of a closed manifold does not admit a subcritical Stein filling by \cite[Proposition~3.9]{BarthGeigesZehmisch19},   $(S^2 \times S^3, \xi_{can})$ is not contactomorphic to any of the subcritically Stein fillable contact manifolds that appeared in the proof of Theorem~\ref{IPsubcrit1}. 
\end{example}

With the symplectic fillings $X_{2k+1}$ of $S^5$ in hand, we are ready to prove Theorem~\ref{felxibleIP}. 

\begin{proof}[Proof of Theorem~\ref{felxibleIP}]
We begin by recalling a result of Lazarev \cite[Theorem~1.1]{Lazarev20a}, the proof of which says that if $\xi$ is a contact structure on $M$ with $c_1(\xi)=0$ that has a flexible filling $W$, then for any two flexible fillings $W_1$ and $W_2$ of $S^5$, the contact structures on $M$ induced by the flexible fillings $W\natural W_1$ and $W\natural W_2$ are distinct provided that $W_1$ and $W_2$ have distinct homologies.  (Here we can think of $W\natural W_i$ as being constructed from $W\cup W_i$ by attaching a Weinstein $1$--handle.) 

We claim that there are flexible Weinstein structures on the manifolds $X_{2k+1}$ and that the contact structure on their boundaries are all  iterated planar. When discussing the flexible Weinstein structure on $X_{2k+1}$ we will denote it $X'_{2k+1}$. Given this claim and our hypothesized $(M,\xi)$ with $c_1(\xi)=0$ that is iterated planar and has a flexible filling $W$, we can consider $W\natural X'_{2k+1}$. All the contact structures on $M=\partial (W\natural X'_{2k+1})$ are distinct since, as noted above, all the $X'_{2k+1}$ have distinct homologies. Note that $\partial X'_{2k+1}$ is diffeomorphic to $S^5$. Moreover, since the contact structures only differ on a $5$--ball and the space of almost contact structures on $S^5$  is connected,  they are all in the same homotopy class of almost contact structures. Finally,  all the contact structures are iterated planar by Theorem~\ref{IPconnect} concerning connected sums, since $\xi$ and $\partial X'_{2k+1}$ are. 

To prove the claim we revisit the description of the Weinstein fillings $X_{2k+1}$. Recall that they are built from  $(DT^*S^2)\times D^2$ by attaching Weinstein $3$--handles along $2k+1$ copies of the zero section of $DT^*S^2$. Also recall that $(\partial (DT^*S^2 \times D^2), \xi)$ is a contact manifold (see Example~\ref{ex: simple})  supported by the open book with page $DT^*S^2$ and monodromy the identity.  If $\lambda$ is the Liouville form on $T^*S^2$, then $d\lambda$ is the symplectic form on $T^*S^2$. The zero section $Z$ of $DT^*S^2$ is Lagrangian and in fact $\lambda|_Z=0$, so it corresponds to a Legendrian sphere on each page of the open book supporting $(\partial (DT^*S^2 \times D^2), \xi)$. We can get $X_{2k+1}$ by attaching Weinstein $3$--handles to $DT^*S^2 \times D^2$ along $2k+1$ copies of this Legendrian sphere. It is well known, see \cite{vanKoert17} and Construction~\ref{theconstruction},  that the resulting open book for $\partial X_{2k+1}$ still has page $DT^*S^2$ but monodromy a composition of $2k+1$ Dehn twists about $Z$. 

Now let $\Lambda$ be a fiber in the bundle $DT^*S^2\to S^2$, and $U$ the boundary of $\Lambda$. So $\Lambda$ is a Lagrangian disk in $(DT^*S^2, d\lambda)$, and $U$ is a Legendrian circle in $(\partial (DT^*S^2), \ker \lambda)$. Moreover, $\Lambda$ intersects $Z$ exactly once and $\lambda=0$ on $\Lambda$.  Let $Y$ be the result of attaching a Weinstein $2$--handle to $DT^*S^2$ along $U$. Notice that there is now an exact Lagrangian sphere $S$ in $Y$ that is the union of $\Lambda$ and  the core of the $2$--handle. Let $\tau_S$ be the Dehn twist about $S$ and $\tau_Z$ be the Dehn twist about $Z$. The open book $(Y,\tau_Z\circ\tau_S)$ is a stabilization of the open book $(DT^*S^2,\tau_Z)$, \cite{vanKoert17}, and thus supports the same contact structure. Notice that $(DT^*S^2, \tau_Z)$ supports the standard contact $S^5$ and in that contact manifold $Z$ is the standard Legendrian unknot. Thus $(Y,\tau_Z\circ \tau_S)$ also supports $(S^5,\xi_{std})$ and in this manifold both $Z$ and $S$ are standard Legendrian unknots.  This should be clear as $Y$ is the plumbing of two copies of $DT^*S^2$ and $S$ and $Z$ are the zero sections of the two copies. Moreover, this open book is the boundary of the Lefschetz fibration obtained from $Y\times D^2$ by attaching two Weinstein $3$--handles, one along $Z$ and one along $S$. The total space of this Lefschetz fibration is a Weinstein  $6$--manifold, which we denote $(E, \omega)$.  Notice that $DT^*S^2$ is obtained by attaching  a Weinstein $2$--handle to the standard symplectic $B^4$ and $Y$ is the result of attaching a Weinstein $2$--handle to $DT^*S^2$. It follows that $(E,\omega)$  is simply the standard symplectic $B^6$, since $Y\times D^2$ is obtained from the standard symplectic $B^6$ by attaching two Weinstein $2$--handles, and the Weinstein $3$--handles, which are attached to $Y \times D^2$ along $Z$ and $S$ to obtain $E$, cancel them. We finally notice that $X_{2k+1}$ has an open book decomposition $(Y, \tau_Z\circ\tau_S\circ\tau_{Z}^{2k+1})$.

In \cite[Section~4.1]{HondaHuang18pre} Honda and Huang defined an operation $S\uplus Z$ on $S$ and $Z$, which produces a new Legendrian sphere that is a stabilization of $Z$. We note that there is a technical hypothesis that $S$ and $Z$ must intersect $\xi$-transversely, that is, their tangent spaces must span the contact hyperplane at an intersection point. But this is clear as the contact hyperplanes are tangent to the page at $S\cap Z$ and the intersection is transverse in the page. Now let $X'_{2k+1}$  be the result of attaching Weinstein $3$--handles to $E$ along $2k+1$  copies of $S\uplus Z$. By \cite[Section~4.1]{HondaHuang18pre} the sphere $S\uplus Z$ is a stabilization of $Z$ which is the standard unknotted Legendrian sphere, thus it is a loose sphere. Hence $X'_{2k+1}$ is a flexible Weinstein manifold that is diffeomorphic to $X_{2k+1}$ because the sphere $S\uplus Z$ is smoothly isotopic to $Z$ (and their framings are the same since they are in $\pi_2(SO(3))$, which is indeed trivial).  
 
Next we claim that the contact structure on the boundary of the Weinstein $6$--manifold $X'_{2k+1}$ is supported by the open book  $(Y, \tau_Z\circ \tau_S\circ \tau^{2k+1}_{S\uplus Z})$. To establish this we need to see that  $S\uplus Z$ is realized by a Lagrangian on a page of the open book  $(Y,\tau_Z\circ \tau_S)$. We show this in Lemma~\ref{lem: Lagrelpage} below. Thus we are done with the proof of Theorem~\ref{felxibleIP} by noticing that $X'_{2k+1}$ is iterated planar by Lemma~\ref{lem: bind}, since the binding $\partial Y$ is a lens space.  
\end{proof}

\begin{lemma}\label{lem: Lagrelpage} Suppose that $S$ and $Z$ are Legendrian $2$-spheres in a contact $5$-manifold $(M^5, \xi)$ which intersect $\xi$-transversely at a point.  Suppose also that $S$ and $Z$ are both Lagrangian  on a page of an open book supporting $(M^5, \xi)$.  Then the stabilized Legendrian sphere $S\uplus Z$ in  $(M^5, \xi)$ can be realized as a Lagrangian on a page of this open book.
\end{lemma}

\begin{proof} 
We ellaborate on the construction of $S\uplus Z$ in the previous proof, since we will need a specific local model in this proof. 
Consider $\R^5$ with the contact form $\alpha=dz-2\mathbf{y}\cdot d\mathbf{x}- \mathbf{x}\cdot d\mathbf{y}$ (where $\mathbf{x}$ is $(x_1,x_2)$ and $\mathbf{y}$ is $(y_1,y_2)$). There is a neighborhood $\mathcal{N}$ of the transverse intersection point between the spheres $Z$ and $S$ that is contactomorphic to a neighborhood of the origin in $(\R^5, \ker \alpha)$ so that $ \mathcal{N} \cap S$ goes to the $\mathbf{x}$-plane and $ \mathcal{N} \cap Z$ to the $\mathbf{y}$-plane. One then embeds $A=S^1\times \R$ into the hyperplane $z=0$  so that it is Legendrian in $(\R^5, \ker \alpha)$ (and also Lagrangian in the symplectic  hyperplane $z=0$) and is asymptotic to $Z$ on one end and $S$ on the other end. Looking at the $S$ end, one can see that $A$ is the $1$-jet of some function $f_S$ over $\mathbf{x}$. Then using a cutoff function $\phi$, the $1$-jet of $\phi f_S$ deforms $A$ so that it agrees with $S$ away from the intersection point, denote the deformed $A$ by $A'$. This can be done for the $Z$ end too. Thus removing disks from $S$ and $Z$ and inserting the portion of $A'$ we have constructed $S\uplus Z$. 

Now in our situation we notice that in this local model, the page of the open book on which $S$ and $Z$ sit can be taken to be the $z=0$ hypersurface. Thus $A$ also sits on the page and $A'$ is the graph of a function $f:A\to \R$ in $\R^5$. We can extend $f$ to all of $\R^4$ so that it is zero outside a neighborhood of the intersection between $S$ and $Z$. Thus we have constructed a Lagrangian $L_{S\uplus Z}$ in the page of our open book and  $S\uplus Z$ is a graph over this Lagrangian. 

Now we can isotope our monodromy map so that this construction takes place in a neighborhood of the intersection point on which the monodromy is the identity. Let $\lambda$ be the Liouville form on the page, so near $S\cap Z$ the contact form is given by $dz+\lambda$ and $S\uplus Z$ being Legendrian implies that $df+\lambda=0$ on $L_{S\uplus Z}$. Now $\lambda_t=\lambda- t d f$ is a family of Liouville forms on the page that induce contactomorphic contact structures on the manifold and $L_{S\uplus Z}$ is Lagrangian for $d\lambda_t$. 

Moreover,  $L_{S\uplus Z}$ is  Legendrian with respect to the contact form induced by $\lambda_1$ and  the graph of $(1-t)f$ is Legendrian with respect to the contact form induced by $\lambda_t$. Thus we have an isotopy from $S\uplus Z$ in our original contact manifold to the Legendrian $L_{S\uplus Z}$ in the contact structure induced by $\lambda_1$, and there $L_{S\uplus Z}$ is also a Lagrangian on the page of the open book.  
\end{proof}

We now turn to Proposition~\ref{fillablebutnoteStein} that shows, contrary to what happens in dimension 3, that in higher dimensions there are iterated planar contact manifolds that are strongly symplectically fillable but not Stein fillable. 

\begin{proof}[Proof of Proposition~\ref{fillablebutnoteStein}]
Let $W_d^4$ be the Stein domain which is the complement of a neighborhood of a symplectic hypersurface $\Sigma_d$ of degree $d > 1$ 
in $\mathbb{CP}^2$ and  let $\tau_d$ denote the fibered Dehn twist on $W^4_d$ along its boundary $\del W^4_d$. Then, according to \cite[page 423]{ChiangDingvanKoert14}, the contact manifold supported by the open book with page $W^4_d$ and monodromy $\tau_d$ is contactomorphic to $(L^5_d, \xi_{0})$, where $L^5_d$ denotes  the lens space $S^5 / \mathbb{Z}_d$ and $\xi_0$ is obtained by taking the quotient of $\xi_{st}$. We conclude by \cite[Theorem 6.3]{ChiangDingvanKoert14} that  $(L^5_d, \xi_{0})$ is of Boothby-Wang type and therefore it is symplectically fillable by the corresponding disk bundle (see \cite[Lemma 3]{GeigesStipsicz10}), but this symplectic filling can not be Stein, since, as was observed in \cite[Example 6.5]{BowdenCrowleyStipsicz14}, for any $d > 1$, the lens space $L_d^5$ does not carry any Stein fillable contact structures at all.

The binding of the above open book is given by the convex boundary of the Stein page $W^4_d$. Note that $\partial W^4_d$ can be described as a circle bundle over $\Sigma_d$ of Euler number $-[\Sigma_d]^2=-d^2,$ and the induced contact structure on $\partial W^4_d$ is of Boothby-Wang type.

For the case $d=2$, we have $W^4_2 \cong DT^*\mathbb{RP}^2$ ($\cong \mathbb{CP}^2 \setminus$ a quadric), and thus  $\partial W^4_2 \cong   ST^*\mathbb{RP}^2 \cong L(4,1)$.  Therefore,  by Lemma~\ref{lem: bind} and the fact that $L(4,1)$ is planar \cite{Schoenenberger07}, we conclude that $(L_2^5, \xi_0)$ is an iterated planar contact $5$--manifold, which is symplectically fillable but not Stein fillable. 
\end{proof}

Turning to overtwisted contact structures we now prove Theorem~\ref{ot}.

\begin{proof}[Proof of Theorem~\ref{ot}]
By the proof of Theorem~\ref{IPsubcrit1}, we know that each of the homotopy classes of almost contact structures mentioned in the theorem is realized by a contact structure $ \xi$ on $M^5$ that is supported by an iterated planar open book. Moreover, a negative stabilization of such an  open book supports   $(M^5 \# S^5, \xi \# \xi_{OT})$ \cite[Theorem 1.1]{CasalsMurphyPresas19}, and $\xi$ and the overtwisted contact structure $ \xi \# \xi_{OT}$ are in the same homotopy class of almost contact structures on $M^5$. Notice that since both $(M^5,\xi)$ and  $(S^5, \xi_{OT})$ are iterated planar, so is $(M^5 \# S^5, \xi \# \xi_{OT})$, by Theorem~\ref{IPconnect}.  
\end{proof}

\begin{example} \label{ex: Barden} ({\em Simply-connected contact $5$--manifolds}) Recall from the introduction that the overtwisted contact structures on the simply connected $5$--manifolds $M_1=S^5$, $M_\infty=S^2\times S^3$, and $X_\infty=S^2\widetilde\times S^3$, are all iterated planar and each of these $5$--manifolds also carry infinitely many iterated planar Stein fillable contact structures. However, we were unable to say anything about $M_k$ for $1<k<\infty$. 

The $5$--manifold $M_k$ carries a Stein fillable contact structure  $\eta_k$, which is described as a contact open book with Stein pages in \cite[Section 7.2]{DingGeigesvanKoert12}. Specifically let $X_k$ be the Stein domain given by the Legendrian surgery diagram depicted in Figure~\ref{figMk}. 
\begin{figure}[htb]{
\begin{overpic}
{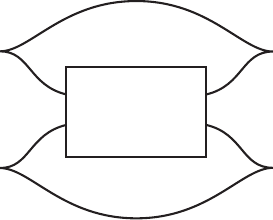}
\put(35, 48){$k$ full twists}
\end{overpic}}
\caption{The Stein domain $X_k$.}
\label{figMk}
\end{figure}
Notice that both knots in the diagram bound Lagrangian disks in $B^4$ and hence there are Lagrangian spheres $S_1$ and $S_2$ in $X_k$ that come from capping these disks off with the core disks of the $2$--handles. Now consider the Stein manifold $X_k\times D^2$ and attach two Stein $3$--handles to each of these spheres. This gives a Stein domain $W_k$ such that $\partial W_k=M_k$. Moreover, we see that the contact structure $\eta_k$ induced on $M_k$ is supported by the open book $(X_k, (\tau_{S_1}\circ \tau_{S_2})^2)$. 

Notice that the intersection form on $X_k$ cannot embed in a diagonal negative definite form for $k>2$. This says the contact structure on $\partial X_k$ is not planar as discussed in Item~\eqref{2} in the introduction. Moreover, since $b_2^0=1$ when $k=2$ we also see that $\partial X_2$ cannot be planar by Item~\eqref{2} as well. Thus $X_k$ does not have a planar Lefschetz fibration and we cannot use the above open book to see that $\eta_k$ is iterated planar. In fact our guess is that $\eta_k$ for $2\leq k<\infty$ is not iterated planar. 

Furthermore, for a simply connected $5$--manifold $M$, homotopy classes of almost contact structures are in one-to-one correspondence with integral lifts of $w_2(M)$ and the correspondence is given by associating to an almost contact structure its first Chern class \cite[Proposition 8.1.1]{Geiges08} (note that this is true for simply connected $5$--manifolds, as pointed out in an erratum; there are more subtleties in the general case). It follows that $M_k$ admits a unique homotopy class of almost contact structures since $H^2(M_k; \Z) = 0$. Thus $M_k$ has a unique overtwisted contact structure. For $1<k<\infty$ it is unclear whether these are iterated planar.
\end{example}

\section{Symplectic cobordisms}\label{capingproceedure}

Eliashberg \cite{Eliashberg04}  proved that any weak symplectic filling of a closed contact $3$--manifold can be symplectically embedded into a closed symplectic $4$--manifold. (Note that an alternate independent proof of this result was obtained by the second author \cite{Etnyre04a}.) 

As discussed in the introduction, Conway and the second author \cite{ConwayEtnyre20} and, independently, Lazarev \cite{Lazarev20b} have shown that caps for contact manifolds can always be constructed (though it is still not clear if weak symplectic fillings of a contact manifold can be embedded in closed symplectic manifolds). But in dimension 3, the caps constructed in \cite{Eliashberg04, Etnyre04a} are more explicit, and so one might hope for new constructions of caps in higher dimensions that more closely follow the construction in dimension 3. 

After reviewing Eliashberg's proof briefly, we discuss some partial generalizations to higher dimensions below. The first step in Eliashberg's proof is the construction of a cobordism $W$ equipped with a symplectic form $\omega$ such that $\partial W = -M \cup N$ with the following properties: $M$ is a concave boundary component of $W$, contactomorphic to the given weakly fillable contact $3$--manifold and $N$ admits a fibration  over $S^1$ such that the  restriction of $\omega|_N$ to each fiber is symplectic.  This cobordism is obtained  by a symplectic $2$--handle attachment along the binding of an open book adapted to the weakly fillable contact $3$--manifold at hand. Then he fills in this symplectic fibration over $S^1$ by a symplectic Lefschetz fibration over $D^2$ to obtain a symplectic cap.

The first step in Eliashberg's program has already been carried out in higher dimensions by D\"orner, Geiges, and Zehmisch  \cite{DornerGeigesZehmisch14}, which we stated as Theorem~\ref{thm: cob1} in the introduction.  In order to construct a  symplectic cap in higher dimensions, one also needs to obtain an  analogue of the second step of Eliashberg's proof, which we formulate as Question~\ref{question: fib}.  

\begin{question} \label{question: fib} 
For $n >1$, assume that $N$ is a $(2n+1)$--dimensional closed manifold which is a boundary component of a symplectic manifold $(W, \omega)$ so that $(N, \omega|_{N})$ admits a symplectic fibration over $S^1$.  Does there exist a symplectic manifold $Z$ so that we can use $Z$ to cap off $N$?
\end{question}

Theorem~\ref{projcap} proved below gives a positive answer to this question in some nontrivial cases. The simplest case when the answer is yes, is when the symplectomorphism group of the fiber $(X, \omega)$ in the symplectic fibration $N \to S^1$ is connected. This is because we can take the monodromy of the symplectic fibration to be the identity and cap off $N$ by $(X \times D^2, \omega+\omega_{st})$, relying on Lemma~\ref{lem: symp}. 

\begin{lemma}\label{lem: symp}
Let $(X, \omega)$ be a closed symplectic manifold and $\phi_i: X\to X$, for $i= 0,1$, be symplectomorphisms that are isotopic through symplectomorphisms. Let $X_{\phi_i}$ denote the mapping torus of $(X, \phi_i)$. Then there is a symplectic manifold $(W,\omega_W)$ such that $$\partial W=-X_{\phi_0} \cup X_{\phi_1},$$
and $\omega_W$ induces the symplectic structures on the fibers of the mapping cylinders $X_{\phi_i}$. 
\end{lemma}

\begin{proof}
Let $\phi_s: X \to X$ be the isotopy of symplectomorphisms and assume that $\phi_s$ is independent of $s$ near $t=0$ and $1$. Consider $X \times \R \times [0,1],$ and coordinates $t\in \R$ and $s\in [0, 1]$ on the last two factors. Notice that $\Omega = \omega+dt\wedge ds$ is a symplectic form on $X\times \R  \times [0,1]$. Consider the map 
\begin{align*}
\Psi :X \times \R \times [0,1] &\to X \times \R  \times [0,1],\\ 
(p,t,s) &\mapsto (\phi_s(p),t+ 1,s).  
\end{align*}
Notice that $\Psi^{*}\Omega =\omega+dt\wedge ds + ds\wedge \gamma$ where $\gamma=\phi_s^*(\iota_{X_s}\omega)$ and $X_s$ is the vector field whose flow gives $\phi_s$ (notice that $X_s$ is zero for $s$ near $0$ and $1$). Now $\Omega_u=(1-u)\Omega+u\Psi^*\Omega=\omega+dt\wedge ds + u\, (ds\wedge \gamma)$ is a path of symplectic forms. Indeed notice that $\gamma$ is closed so the last term is $u\, d(s\gamma)$. Applying Moser's method to this family gives a flow generated by the vector field $sX_s$. Let $\Phi:X \times \R \times [0,1] \to X \times \R  \times [0,1]$ be time $1$ map of this flow. Notice that since $sX_s$ is tangent to $X$, and $X_s  = 0$  for $s$ near $0$ and $1$, the map $\Phi$  has the form $\Phi(p,t,s)=(f_s(p),t,s)$ and is the identity near $s=0$ and $1$. Now $\Phi\circ \Psi:X \times \R \times [0,1] \to X \times \R  \times [0,1]$ is a symplectomorphism of $\Omega$ and generates an action of $\Z$ on $X\times \R \times [0,1]$ by symplectomorphisms. So the quotient space $W$ of $X \times \R \times [0,1]$ by this $\Z$--action inherits a symplectic structure $\omega_W$ from $\Omega$. Clearly this is the desired cobordism carrying a symplectic structure.
\end{proof}

We now study when we can glue two symplectic cobordisms along a boundary component that is a symplectic bundle over a circle. To that end we recall that a \dfn{stable Hamiltonian structure} on a $(2n+1)$--dimensional oriented manifold $M$ is a pair $(\omega,\lambda)$ where $\omega$ is a closed $2$-form and $\lambda$ is a $1$-form such that $\lambda\wedge \omega^n>0$ and the kernel of $\omega$ is contained in the kernel of $d \lambda$. The \dfn{Reeb vector field} of such a structure is the unique vector field $R$ such that $R$ is in the kernel of $\omega$ and $\lambda(R)=1$. Moreover, given a stable Hamiltonian structure $(\omega, \lambda)$ on $M$ we have its \dfn{symplectization}, which is the symplectic manifold $((-\epsilon, \epsilon)\times M, \omega + d(t\lambda))$ for small $\epsilon$. If $(\omega, \lambda)$ and $(\omega', \lambda')$ are stable Hamiltonian structures on $M$ and $M'$, respectively,   then a diffeomorphism $f:M' \to M$ satisfying $f^*\omega=\omega'$ and $f^*\lambda=\lambda'$ induces a symplectomorphism between their symplectizations. 

If $p:M\to S^1$ is a bundle over $S^1$ and $\omega$ is a closed $2$--form on $M$ that is symplectic on each fiber, then $(\omega, \lambda=p^*d\theta)$ is a stable Hamiltonian structure. We note that such bundles have a standard form. Choose a fiber $F_0=p^{-1}(\theta_0)$ and let $\omega_0$ denote the restriction of $\omega$ to ${F_0}$. Let $\phi:F_0\to F_0$ be the first return map of the flow of the Reeb vector field $R$. Since it is clear that the flow of $X$ preserves $\omega$, we see that $\phi$ is a symplectomorphism of $(F_0,\omega_0)$. We can now form the mapping cylinder $C_\phi$ by taking the quotient of $F_0\times \R$ by the action of $\Z$ generated by $(p,t)\mapsto (\phi(p),t-2\pi)$. The pull-back of $\omega_0$ to $F_0\times \R$ (which we still denote by $\omega_0$) is invariant under this action, so induces a form (still denoted $\omega_0$) on $C_\phi$. Moreover, the from $dt$ descends to a $1$--form, which we call $d\theta$, on  $C_\phi$ and hence $(\omega_0,d\theta)$ is a stable Hamiltonian structure on $C_\phi$. One may use the identification of $F_0$ with a fiber of $p:M\to S^1$ and the flow of $R$ to construct a diffeomorphism $f:C_\phi\to M$ that pulls back $(\omega, d\theta)$ to $(\omega_0,d\theta)$. The upshot is that all symplectic bundles over $S^1$ are equivalent, as stable Hamiltonian structures, to a mapping cylinder. 

We now observe that symplectic bundles over circles in the boundary of a symplectic manifold have standard neighborhoods. This is essentially Lemma~10 from \cite{DornerGeigesZehmisch14}.

\begin{lemma}\label{nbhdofSHS}
Let $(W,\omega)$ be a symplectic manifold with a boundary component $M$ such that $p:M\to S^1$ is a bundle and $\omega$ is a symplectic form on each fiber of $p$. Let $\omega'$ denote the restriction of $\omega$ to $M$. Then there is a neighborhood $(-\epsilon, 0]\times M$ of $M$ in $W$ such that $\omega$ takes the form 
$
\omega' + dt\wedge p^*d\theta. 
$ In particular, a neighborhood of $M$ in $W$ is symplectomorphic to a piece of the symplectization of the stable Hamiltonian structure $(\omega',\lambda= p^*d\theta)$  on $M$. 
\end{lemma}

\begin{remark}
If in the statement of the Lemma \ref{nbhdofSHS} we had $-M$ being a boundary component of $W$, then the conclusion would be a neighborhood $[0,\epsilon)\times M$ with the same properties. 
\end{remark}
\begin{proof}[Proof of Lemma~\ref{nbhdofSHS}]
Choose any identification of a neighborhood of $M$ in $W$ with $(-\delta,0]\times M$, so that we have a projection from this neighborhood to $S^1$. Denote the pull-back of $d\theta$ by $d\theta$ and let $X$ be the unique vector field on the neighborhood such that $\iota_X\omega=d\theta$. We claim that $X$ transversely points out of $M$. To see this let $R$ be the Reeb field of $(\omega',\lambda)$ and compute $\omega(X,R)=\lambda(R)=1$. Therefore, $X$ cannot be tangent to $M$ since $R$ is in the kernel of $\omega$ restricted to $M$.  Since the flow of $X$ preserves $\omega$ we can use the backwards flow of $X$ to create the desired neighborhood. 
\end{proof}

The above results will be used in the next section to study caps for some contact manifolds supported by projective open books. For now we turn to the construction of symplectic cobordisms between contact manifolds. In particular, we prove Theorem~\ref{thm: cob}, which says: 
let $(M,\xi)$ be a $(2n+1)$--dimensional closed contact manifold supported by an open book decomposition with page $F$, binding $B=\partial F$ and monodromy $\phi$. Suppose that $X$ is an exact symplectic cobordism such that $\partial X=-B \cup B'$. Then there exists a strong symplectic cobordism $$W=M \times [0, 1] \bigcup_{B\times D^2 \times \{1\} =\partial_{-}X \times D^2} X\times D^2$$  (after rounding corners) such that $\partial W=\partial_{-}W \cup \partial_{+}W$ with $\partial_{-}W=-M$ and $\partial_{+}W=M'$ where $(M', \xi')$ is supported by the open book with page $F \cup X$, binding $B'$ and monodromy $\phi'$ which is identity on $X$ and agrees with $\phi$ on $F$.

\begin{proof}[Proof of Theorem~\ref{thm: cob}]
If  $d\lambda$ is the exact symplectic form on the given cobordism $X$,  then $d \lambda'$ is an exact symplectic form on $X \times D^2$, where $\lambda' = \lambda+x\, dy-y\, dx$.  Let $V$ denote the Liouville vector field of $\lambda$ on $X$.  It follows that $$Z=V+\dfrac{1}{2}\left(x\frac{\partial}{\partial x}+ y\frac{\partial}{\partial y}\right)$$ is the Liouville vector field for $\lambda'$  on $X \times D^2$. Notice that
 $$\partial(X\times D^2)=(X\times \partial D^2) \cup (\partial_{-}X\times D^2) \cup (\partial_{+}X\times D^2).$$

We consider $M \times [0, 1]$ as a part of  the symplectization of $(M, \xi)$ and using $Z$  we glue the symplectic manifold $(X \times D^2, d\lambda')$ to $M \times [0, 1]$ by identifying  $B\times D^2 \times \{1\}$ with $\partial_{-}X\times D^2$ as follows. The gluing of $X \times D^2$  can be viewed as a generalized handle attachment, which we depicted in Figure~\ref{cobord} below. In particular,  the dark green on the right is obtained by the flow of $Z$ applied to the lower boundary, $\partial_{-}X \times D^2$ (actually we apply it to $\partial_{-}X \times D'$ where $D'$ is a slightly smaller disk in $D^2$ as indicated in Figure~\ref{cobord}). Therefore,  it is part of the symplectization of $\partial_{-}X \times D^2$, and thus it can be identified with part of the symplectization of $B\times D^2=\partial_{-}X\times D^2$ in the symplectization of $M$. As a result (after rounding corners) we get the symplectic cobordism $W$ indicated in the theorem. 

\begin{figure}[htb] 
\includegraphics[scale=.3]{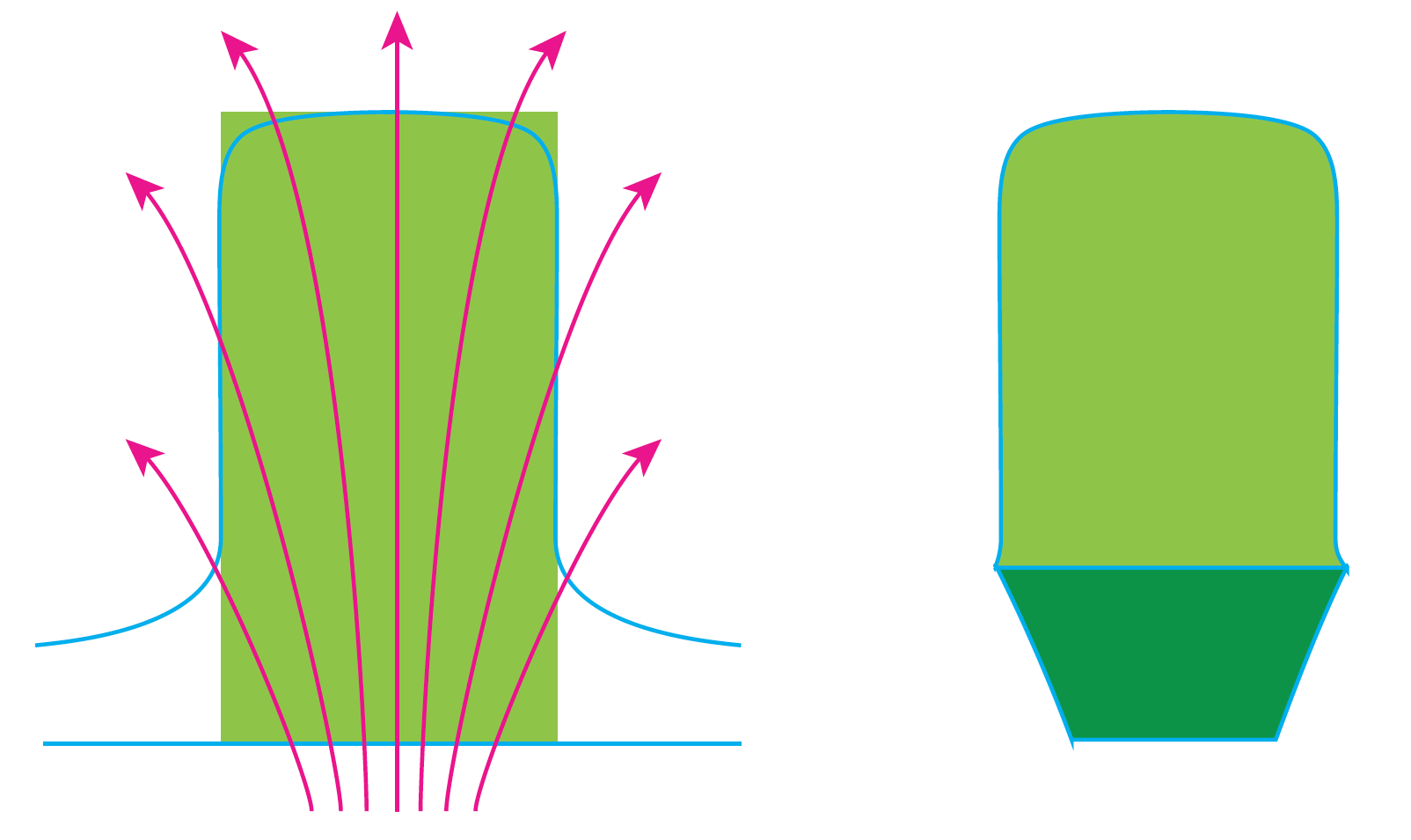}
\caption{The "handle" attachment. The green on the left is $X \times D^2$. The red lines represent the flow lines of the Liouville field $Z$. The blue is the smoothed boundary of the handle we want to attach. The actual handle is depicted on the right.}
\label{cobord}
\end{figure}

To finish the proof we need to check that the upper boundary $(M', \xi')$ of the cobordism  $W$ is supported by the open book described in the theorem. First we observe that the fibration of  the complement of $B \times D^2$ in $M$ extends to a fibration of the complement of $B'$ in $M'$. This is because while attaching the ``handle" $X \times D^2$, we perform a surgery on $M$ removing $B \times D^2$ and gluing in $(X \times \partial D^2) \cup (\partial_+X\times D^2)$ so that $B \times \{p\}$ is identified with  $\partial_- X \times \{p\}$ for each $p \in \partial D^2$.  This shows that $M'$ admits an open book with page $F \cup X$, binding $B'$ and monodromy $\phi'$ which is identity on $X$ and agrees with $\phi$ on $F$. The contact structure $\xi'$ on $M'$ can be given by the kernel of the $1$-forms $\lambda + x\, dt -y\, dx$ on $\partial_+ X \times D^2$, $\lambda$ on $X\times \partial D^2$, and the original contact form on $M$ minus $B\times D^2$, respectively. These contact forms are clearly supported by the pages of the open book. Since the Liouville field $Z$ remains transverse to the upper boundary as we ``round corners" to construct the cobordism the compatibility is preserved. 
\end{proof}

\section{Projective contact manifolds}\label{ProjectiveOB}
We begin this section with a quick proof of Theorem~\ref{ipisp} which says that an iterated planar contact $5$--manifold is projective. Recall that an open book for a $(2n+1)$--manifold is called projective if its Weinstein page embeds as a convex domain in $\C P^n\#_k \overline{\C P^n}$ for some $k$ and  a contact manifold supported by a projective open book is called projective.
\begin{proof}[Proof of Theorem~\ref{ipisp}] The Weinstein page of an iterated planar open book supporting a contact $5$--manifold $(M^5, \xi)$  is a symplectic filling of the planar  binding of the open book.
In the proof of Theorem~4.1 in \cite{Etnyre04b}, the second author has shown that any symplectic filling of a planar contact manifold embeds in a blowup of $S^2\times S^2$, and by possibly blowing up again we can assume that there is at least one blowup. But of course such a manifold is symplectomorphic to $\C P^2$ blown up some number of times, which shows that the iterated planar open book above is projective, and hence $(M^5, \xi)$  is projective.
\end{proof}

We now turn to building caps for some projective contact $5$--manifolds. Specifically, we will prove Theorem~\ref{projcap} which says that if $(M^5,\xi)$ is supported by an open book whose page embeds as a convex domain in $\C P^2\#_n \overline{\C P}^2$ for $n\leq 4$, then $(M,\xi)$ has a symplectic cap that contains an embedded $\C P^2\#_n \overline{\C P}^2$.

\begin{proof}[Proof of Theorem~\ref{projcap}]
Suppose $(M^5,\xi)$ is supported by the open book $(X,\phi)$ where $X$ embeds as a convex domain in $F=\C P^2\#_n\overline{\C P}^2$ for $n\leq 4$. Then we can apply Theorem~\ref{thm: cob1}, using the symplectic cap  $Y= \overline{F \setminus X}$, to build a symplectic cobordism $W$ from $(M,\xi)$ to $M'$, where $M'$ is a symplectic $F$--bundle over $S^1$. Let $\Phi:F\to F$ be the monodromy of this bundle. 

In \cite[Theorem~1.8]{LiWu12}, Li and Wu proved that for any symplectomorphism $f : F \to F$, there are Lagrangian spheres in $F$ such that $f$ acts the same on homology as a composition of Dehn twists about these spheres. In \cite[Theorem~1.1]{LiLiWu15}, Li, Li, and Wu showed that for $\C P^2\#_n \overline{\C P}^2$ with $n\leq 4$, if a symplectomorphism acts trivially on the homology of $F$ then it is symplectically isotopic to the identity map. This of course implies that $f$ is symplectically isotopic to a composition of Dehn twists about the Lagrangian spheres.

Let $S_i$ denote the Lagrangian spheres one gets when this construction is applied to $\Phi^{-1} : F \to F$ and $\Psi : F \to F $ denote the composition of Dehn twists about the $S_i$.  Next we apply Lemma~\ref{lem: symp} to the isotopy from $\Phi$ to $\Psi^{-1}$ to extend the above symplectic cobordism $W$ to a cobordism $W'' = W \cup W'$ that goes from $M$ to $M''$, which is a symplectic $F$--bundle over $S^1$ with monodromy $\Psi^{-1}$.

Now one can take the trivial Lefschetz fibration $F\times D^2$ and attach symplectic $3$--handles along the spheres $S_i$ in different fibers of $F\times S^1$ to get a symplectic Lefschetz fibration $E \to D^2$ whose monodromy is $\Psi$ (see Construction~\ref{theconstruction}). It follows that  $ \del E$ is an $F$--bundle over $S^1$ with monodromy $\Psi$.

Finally, since the diffeomorphism from $M'' \subset  \partial W''$ to $\del E$ given by 
$$
F\times[0,1]/\sim_{\Psi^{-1}} \;\to \; F\times [0,1]/\sim_{\Psi} 
$$ $$
(p,t)\mapsto (p,1-t)
$$
is a symplectomorphism on the fibers, we can use Lemma~\ref{nbhdofSHS} to glue $W''$ and $E$ together along $M''\cong \del E$ to get the desired symplectic cap for $(M,\xi)$. 
\end{proof}

\def\cprime{$'$} \def\cprime{$'$}

\end{document}